\documentclass[a4paper]{amsart}
\usepackage[utf8]{inputenc}

\usepackage{amssymb}
\usepackage{amsthm,mathtools}
\usepackage{listings}
\usepackage[table]{xcolor}

\usepackage{makecell}
\usepackage{tabularx}
\usepackage{adjustbox} 
\usepackage{array} 

\usepackage{hyperref}

\definecolor{lightgray}{gray}{0.9}
\definecolor{lightblue}{RGB}{220,230,241}
\definecolor{lightgreen}{RGB}{222,237,227}
\definecolor{lightorange}{RGB}{255,235,205}
\definecolor{lightpurple}{RGB}{235,222,240}
\definecolor{lightred}{RGB}{255,220,220}

\usepackage[active]{srcltx}
\usepackage{color}
\usepackage{enumerate}
\usepackage[all]{xy}
\usepackage{enumitem}
\setlistdepth{9}
\usepackage{multirow}

\setlist[itemize,1]{label=$\bullet$}
\setlist[itemize,2]{label=$\bullet$}
\setlist[itemize,3]{label=$\bullet$}
\setlist[itemize,4]{label=$\bullet$}
\setlist[itemize,5]{label=$\bullet$}
\setlist[itemize,6]{label=$\bullet$}
\setlist[itemize,7]{label=$\bullet$}
\setlist[itemize,8]{label=$\bullet$}
\setlist[itemize,9]{label=$\bullet$}

\renewlist{itemize}{itemize}{9}
\usepackage{float}
\usepackage[all]{xy}

\usepackage{adjustbox}

\renewcommand{\k}{\Bbbk}

\newcommand{\nc}{\newcommand}

\newcommand{\ot}{\otimes}

\newcommand{\ydh}{{}^{H}_{H}\mathcal{YD}}

\newcommand{\rg}{\rangle}
\renewcommand{\lg}{\langle}
\def\qb{\mathfrak{q}}
\newcommand{\bq}{\mathfrak{q}}

\newcommand{\id}{\operatorname{id}}
\newcommand{\gr}{\operatorname{gr}}

\newcommand{\Ss}{\mathcal{S}}

\newcommand{\Alg}{\operatorname{Alg}}

\newcommand{\ord}{\operatorname{ord}}

\newcommand{\Cleft}{\operatorname{Cleft}}

\renewcommand{\mod}{\operatorname{-mod}}

\newcommand{\N}{\mathbb{N}}
\newcommand{\Z}{\mathbb{Z}}

\newcommand{\G}{\mathbb{G}}
\newcommand{\s}{\mathbb{S}}
\newcommand{\tx}{\texttt{x}}
\newcommand{\tb}{\texttt{b}}
\newcommand{\I}{\mathbb{I}}

\newcommand{\BB}{\mathbb{B}}
\newcommand{\EE}{\mathbb{E}}

\newcommand{\Bq}{\mathfrak{B}}
\newcommand{\mE}{\mathcal{E}}

\newcommand{\eps}{\epsilon}
\newcommand{\rhu}{\rightharpoonup}

\setlistdepth{11}

\def\pf{\begin{proof}}
	\def\epf{\end{proof}}

\def\bs{\boldsymbol}

\renewcommand{\H}{\operatorname{H}}
\renewcommand{\Z}{\operatorname{Z}}
\newcommand{\B}{\operatorname{B}}

\newcommand{\Dchaintwo}[3]{\xymatrix@C-4pt{\overset{#1}{\underset{\ }{\circ}}\ar
		@{-}[r]^{#2}
		& \overset{#3}{\underset{\ }{\circ}}}}

\nc{\ben}{\begin{enumerate}[(i)]}
	\nc{\een}{\end{enumerate}}

\numberwithin{equation}{section}

\theoremstyle{plain}
\newtheorem{theorem}{Theorem}[section]
\newtheorem{lemma}[theorem]{Lemma}

\newtheorem{proposition}[theorem]{Proposition}
\newtheorem{corollary}[theorem]{Corollary}

\newtheorem{claim}[theorem]{Claim}
\newtheorem{definition}[theorem]{Definition}

\theoremstyle{remark}
\newtheorem{remark}[theorem]{Remark}
\newtheorem*{acknowledgement*}{Acknowledgement}

\newtheorem{example}[theorem]{Example}

\allowdisplaybreaks

\title[Hopf 2-cocycles of type $A_2$]{Hopf 2-cocycles of type $A_2$}

\author{José Ignacio Sánchez}
\address{J.I.S.: FaMAF-CIEM (CONICET),
		Medina Allende S/N,
		Universidad Nacional de C\'ordoba,
		Ciudad Universitaria, C\'ordoba (X5000HUA),
		República Argentina. }
\email{jose.ignacio.sanchez@mi.unc.edu.ar}

\thanks{The work was partially supported by CONICET,
	FONCyT-ANPCyT, Secyt (UNC)}


\begin{document}
	
	\begin{abstract}
We compute the Hopf 2-cocycles involved in the classification of pointed Hopf algebras of diagonal type $A_2$. 
When the quantum Serre relations are deformed, we characterize those cocycles that can be recovered from Hochschild cohomology, via exponentiation.
We identify some hypotheses that allow us to present general formulas that apply in our setting.
	\end{abstract}
	
	\maketitle
	
	\section{Introduction}
The classification program of finite-dimensional pointed Hopf algebras over an abelian group has remained an active area of research for the last two decades.

A decisive tool to tackle this problem has been the lifting method, developed by Andruskiewitsch and Schneider \cite{AS98}.
In this sense, once a finite abelian group $G$ is fixed, two decisive steps of this method are: to determine the family $\mathfrak{F}_G$
of all Nichols algebras $\Bq$ (of diagonal type) of finite dimension in the category $\ydh$ of Yetter–Drinfeld modules over $H=\k G$; and then to compute all their liftings or deformations over $H$, that is,
those Hopf algebras $L$ whose coradical is isomorphic to $H$,  and whose graded Hopf algebra associated to the coradical filtration, $\gr L$, coincides with the Radford biproduct $\Bq\#H$ for some $\Bq$ in $\mathfrak{F}_G$.

After the classification of Nichols algebras of diagonal type with finite root system \cite{H09} and their explicit presentations \cite{An1,An2} became available, only the computation of the deformations remained.
In this direction, in \cite{AAGMV} a technique was developed to produce liftings via cocycle deformations. This technique was decisive to determine all the liftings
of Nichols algebras of Cartan type $A_\theta$ in \cite{AAG}, and later in \cite{AnG}, the classification of liftings of Nichols algebras of diagonal type over an arbitrary abelian group was completed.

In this context, the results in \cite{AnG} establish that if $L$ is a pointed Hopf algebra over an abelian group algebra $H$ and $\gr L\simeq \Bq\# H$, then $L$ is a cocycle deformation of $A=\Bq\# H$.
That is, $L=A$ as coalgebras, and the multiplication $m_L\colon L\ot L\to L$ in $L$ is obtained by deforming the multiplication $m_A\colon A\ot A\to A$ via the convolution $m_L=\sigma\ast m_A\ast\sigma^{-1}$. Here, $\sigma\colon A\ot A\to \k$ is a Hopf 2-cocycle for $A$, that is, it is convolution-invertible and satisfies
\begin{align*}
	\sigma(x_{(1)}, y_{(1)}) \sigma(x_{(2)} y_{(2)}, z) &=
	\sigma(y_{(1)}, z_{(1)}) \sigma(x, y_{(2)}z_{(2)}), \quad x,y,z\in A.
\end{align*}


An analogous situation arises in all known examples of finite-dimensional pointed Hopf algebras over $H=\k G$, where $G$ is not necessarily abelian, see~\cite{GaM,GIM,GIV1,GIV2}. 
Furthermore, the same behavior is observed for copointed Hopf algebras over $H=\k^G$, cf.~\cite{AV2,FGM,GIV2}.

These facts follow from the definition of a suitable cleft object $E$ for $A$ such that $L\simeq L(E,A)$, the left Hopf algebra introduced in \cite{Sc96}. Thus, $L$ is a cocycle deformation of $A$ and the Hopf 2-cocycle is bypassed.
Recall that a cleft object for $A$ is a certain $A$-comodule algebra $E$ endowed with a convolution-invertible comodule isomorphism $\gamma\colon A \to E$.
The remarkable fact is that once $\gamma$ is described, one can define a Hopf 2-cocycle for $A$ via
\begin{align}\label{eqn:sigmagamma-intro}
	\sigma(a,b)=\gamma(a_{(1)})\gamma(b_{(1)})\gamma^{-1}(a_{(2)}b_{(2)}), \qquad a,b\in A.
\end{align}
Thus, $L\simeq L(E,A)$ turns out to be a cocycle deformation of $A$ by the Hopf cocycle $\sigma$ associated with $\gamma$ as in \eqref{eqn:sigmagamma-intro}.


A key feature is that these data are completely determined by their counterparts in the braided setting. 
Indeed, $(E,\gamma)$ factors as $E\simeq \mE\# H$ and $\gamma=\gamma_{\mE}\#\id_H$ in such a way that $(\mE,\gamma_{\mE})$ is a cleft object for $\Bq$ in $H\mod$. 
In turn, $\sigma=\sigma_{\Bq}\#\eps$, where $\sigma_{\Bq}\colon \Bq\ot\Bq\to \k$ is $H$-linear braided Hopf 2-cocycle, see \S\ref{sec:braided-context}.

Given the state of the art we have mentioned, it is clear that the classification problem of finite-dimensional (co)pointed Hopf algebras is intimately connected with Hopf 2-cocycles, so the study and description of these in each case is an interesting problem.

\subsection{Some advances}
One of the first approaches to the computation of cocycles, already present in \cite{GrM}, is based on Hochschild cohomology. The central idea is that, under certain conditions, a Hochschild $2$-cocycle $\eta\colon A\ot A\to\k$ can give rise to a Hopf $2$-cocycle through the exponential
$\sigma := e^{\eta} = \sum_{k \geq 0} \frac{\eta^{\ast k}}{k!}$.

This procedure, developed in detail in \cite{GaM}, made it possible to obtain explicit cocycles associated with the classification of pointed Hopf algebras over dihedral groups $\mathbb{D}_m$ with $m=4t\geq12$, over the symmetric group $\s_3$, and some families over $\s_4$. However, the same work shows that not all cocycles can be obtained in this way, which highlights the existence of a class of “pure” cocycles, inaccessible to this method. The copointed case over $\mathbb{D}_m$ with $m=4t\geq12$ is treated in \cite{FGM}, where the Hopf cocycles are also constructed through exponentials of Hochschild $2$-cocycles.

On the other hand, in \cite{GIS1} the authors develop a theoretical framework to address the problem of computing the Hopf cocycles arising in the liftings obtained by applying the technique given in \cite{AAGMV}.
Moreover, a complete description of the Hopf cocycles associated with the deformations of Nichols algebras of Cartan type $A_2$ with parameter $q=-1$ is provided, showing that these are generically pure. Similarly, following this strategy, in \cite{GIS2} they recover and describe the Hopf cocycles associated with the (co)pointed deformations over $\s_3$, reaching similar conclusions about their purity. Thus, the limitation of Hopf cocycles that arise as exponentials of Hochschild $2$-cocycles was made evident.


\subsection{Results}
We provide the explicit description of the Hopf 2-cocycles involved in the liftings of Cartan type $A_2$, with parameter $q$, a primitive root of unity of order $N\geq 3$, {\it cf.}~\cite{GIS1}. 
In this direction, in Section \ref{sec:Good-Setting} we identify a framework that leads to general formulas for the section $\gamma$ and the associated cocycle $\sigma$, see Theorem \ref{thm:H12}.

We apply these technical results to compute the Hopf cocycles involved in the classification of pointed Hopf algebras of type $A_2$.
More precisely,
let 
\[
\qb=\left(\begin{smallmatrix}
	q&q_{12}\\q_{21}&q
\end{smallmatrix}\right)
\]
be a braiding of Cartan type $A_2$, where $q$ is a primitive root of unity of order $N\geq 3$, and consider a principal YD-realization of $\qb$ in $\ydh$ for a semisimple Hopf algebra $H$. When $N=2$, these cocycles were characterized in \cite{GIS1}.

Let $\Bq_\bq$ be the Nichols algebra associated to $\qb$. Recall that the deformations of $\Bq_\qb$ over $H$ are obtained as the left Hopf algebras
$\mathfrak{u}_\qb(\bs\lambda)= L(\mE_{\bs\lambda}\#H, \Bq_\qb\# H)$, with $\mE_{\bs\lambda}$ a $\Bq_\qb$-cleft object in $H\mod$, where $\bs\lambda=(\lambda_1,\lambda_2,\lambda_{12},\lambda_{112},\lambda_{122})\in\k^5$ with certain restrictions (depending on $H$), see \S \ref{sec:liftA2}.
This splits the deformations into generic liftings (where $\lambda_{112}=\lambda_{122}=0$) and atypical liftings (where $N=3$ and $\bq=\begin{psmallmatrix}
	q&q\\q &q
\end{psmallmatrix}$), see Lemma \ref{lem:AAG}. We deal with each case separately, in Sections \ref{sec:general} and \ref{sec:N3} respectively.
For each case, we:
\begin{itemize}
	\item define an explicit section $\gamma_{\bs\lambda}\colon\Bq_\bq\to\mE_{\bs\lambda}$, see Corollaries \ref{cor:gamma-formula-QS-no-deformed} and \ref{cor:section-QS-deformed},
	\item recover the associated braided Hopf 2-cocycle $\sigma_{\bs\lambda}$ and compute its values on a basis of $\Bq_\qb$, see Theorems \ref{thm:cocycle-generic} and \ref{thm:cocycle-quantum-serre-deformed}.
\end{itemize}

On the other hand, it is important to mention that we have learned from our previous work in \cite{GIS1,GIS2} that the range of Hopf cocycles that can be reached via this procedure is quite scarce. Hence we do not pursue this direction here in full generality.
We do, however, investigate this for the atypical liftings, as this case presents a (subtle) novelty with respect to those in loc.cit., namely that the quantum Serre relations can be deformed. We know from \cite{AAG} that this characteristic has a big impact on the combinatorics within the presentation of the deformations.
Nevertheless, we anticipate the reader that the space of Hopf cocycles obtained via exponentiation is again very limited. In fact, we cannot even reach some of the mono-parametric cocycles (when the parameter is associated to a quantum Serre relation).



The following example summarizes our results in the atypical case, see Theorem \ref{thm:cocycle-quantum-serre-deformed} and Proposition \ref{pro:Conclusion}.
\begin{example}\label{exa:intro}
	Fix $N=3$, $\bq=\begin{psmallmatrix}
		q&q\\q &q
	\end{psmallmatrix}$. 
	The nonzero values {\it in the first rows} of a cocycle $\sigma_{\bs\lambda}\colon \Bq\ot\Bq\to\k$ are given by the table:
	\begin{center}
			\begin{tabular}{|c|c|c|c|c|c|c|}
				\hline
				$\sigma_{\bs\lambda}$  & $x_2^2$ & $x_{12}$ & $x_1^2$ & $x_2x_{12}^2$ & $x_2x_{12}x_1^2$& $x_{12}^2x_1$  \\
				\hline
				$x_2$  & $\lambda_2$ & $0$ & $0$ & $(q-q^2)\lambda_2\lambda_{112}$ & $(q-q^2)\lambda_2\lambda_1$ & $3q^2\lambda_2\lambda_1$ \\
				\hline
				$x_1$  & $\lambda_{122}$ & $\lambda_{112}$ & $\lambda_1$ & $ (q-q^2)\lambda_{112}\lambda_{122}+\lambda_{12}$ & $ (q-q^2)\lambda_1\lambda_{122}$ & $3q^2\lambda_1\lambda_{122}$\\
				\hline
			\end{tabular}
		\end{center}
		Moreover, 
		$\sigma_{\bs\lambda}$ is an exponential of a Hochschild 2-cocycle if and only if either:
		\begin{enumerate}[leftmargin=*]
			\item[(a)] $\lambda_{112}=\lambda_{122}=0$ and at most a single parameter $\lambda_p\in\{\lambda_1,\lambda_2,\lambda_{12}\}$ is non-zero.
			\item[(b)] $\lambda_{112}\lambda_{122}\neq 0$ and $\lambda_1=\tfrac{1}{3}\tfrac{\lambda_{112}^2}{\lambda_{122}}$, $\lambda_2=\tfrac{1}{3}\tfrac{\lambda_{122}^2}{\lambda_{112}}$, $ \lambda_{12}=\tfrac{1}{3}(q^2-q)\lambda_{112}\lambda_{122}$.
		\end{enumerate} 
	\end{example}
	
	We use computer program \cite{GAP} for come of our calculations, the corresponding scripts are uploaded in the author's repository in  \href{https://github.com/JoseIgnacio25/Hopf-cocycles-of-Cartan-type-A2}{github}.
	
	Consequently, together with \cite{GIS1}, we obtain the complete description of the Hopf cocycles associated with the liftings of Cartan type $A_2$. We believe that this result could serve as the base step to advance toward the case of Cartan type $A_\theta$ with $\theta\geq 3$, and subsequently to the other Cartan types.

\subsection*{Acknowledgments}
The author would like to sincerely thank his Ph.D. advisor Agustín García Iglesias for his continuous guidance throughout this work. His careful reading of the manuscript, thoughtful feedback, and help with editing were extremely valuable.

\section{Preliminaries}
We work over an algebraically closed field $\k$ of characteristic zero. 
For $\theta\in\N$, we set $\I_\theta\coloneqq\{1,\dots,\theta\}\subset\N$ and $\I_\theta^\circ=\I_\theta\cup\{0\}$.
Let $\mathbb{G}_\theta$ denote the group of $\theta$-th roots of unity in $\k$ and $\mathbb{G}_\theta'$ the subset of primitive roots.

Let $(H,m,\Delta)$ be a Hopf algebra with bijective antipode. We use Sweedler’s notation for comultiplication and coactions, and denote the counit by $\varepsilon$. 
We write $G(H)$ for the set of group-like elements of $H$, and $P(H)$ for its primitive elements. We also denote by $U(H^\ast)$ the group of convolution unit of $H$.
Let $\ydh$ denote the (braided) category of Yetter-Drinfeld modules over $H$.

If $B$ is any graded (connected) Hopf algebra and $a\in B$ is a homogeneous element of positive degree, then we define
$\underline{\Delta}(a)\coloneqq \Delta(a)-a\ot 1-1\ot a$. 
We set $\underline{\Delta}_{|B_0}=0$ and write $\underline{\Delta}(a)=a_{\underline{(1)}}\ot a_{\underline{(2)}}$.

\subsubsection*{Gaussian binomial coefficients}\label{gaussiano}
Let $\mathbb{Z}[\bs q]$ be the polynomial algebra in $\bs q$, and consider the gaussian coefficients
\begin{align*}
	\binom{n}{k}_{\bs q}=\frac{(n)!_{\bs q}}{(k)!_{\bs q}(n-k)!_{\bs q}},
\end{align*}
where $(n)!_{\bs q}=(n)_{\bs q}(n-1)_{\bs q}\dots (2)_{\bs q}(1)_{\bs q}$ and $(n)_{\bs q}=1+{\bs q}+\dots+{\bs q}^{n-1}$, for $n\in \N$, $1\leq k\leq n$. As usual, we define $(0)!_{\bs q}=1$.

Moreover, we have the identities for $1\leq k\leq n$,
\begin{align}\label{eqn:binomial-gaussian-identities-1}
	&{\bs q}^k\binom{n}{k}_{\bs q}+\binom{n}{k-1}_{\bs q}=\binom{n}{k}_{\bs q}+{\bs q}^{n+1-k}\binom{n}{k-1}_{\bs q}=
	{\bs q}^k\binom{n+1}{k}_{\bs q}.
\end{align}
By induction on $n$, using \eqref{eqn:binomial-gaussian-identities-1}, we can see that $\binom{n}{k}_{\bs q}\in\mathbb{Z}[\bs q]$. Now suppose that $B$ is an associative algebra over $\k$ and $q\in\k$; then the specialization of the polynomial $\binom{n}{k}_{\bs q}$ is denoted by $\binom{n}{k}_q$, and we have the following: if $x,y\in B$ $q$-commute, that is, $xy=qyx$, then for $n\in\N$ the following $q$-binomial formula holds:
\begin{align}\label{eqn:binomial-formula}
	(x+y)^n=\sum_{k=0}^n\binom{n}{k}_qy^kx^{n-k}.
\end{align}
Note that if $q$ is a primitive $N$-th root of unity, then $\binom{N}{k}_q=0$ for each $1\leq k< N$. See \cite[Appendix]{AS00} for further details.

\subsection{Cleft objects and Hopf 2-cocycles}\label{sec:cocycles} 
A (right) {\it cleft object} for $H$ is a right $H$-comodule algebra $E$ such that $E^{coH}=\k$ and there exists a convolution-invertible comodule isomorphism $\gamma\colon H\to E$. This map is called a {\it section} when $\gamma(1)=1$.
The set of $H$-cleft objects, up to isomorphism, is denoted by $\Cleft(H)$.

A {\it Hopf 2-cocycle} for $H$ is a convolution-invertible linear map $\sigma\colon H\ot H\to \k$ satisfying
\begin{align}\label{eqn:cocycle-condition}
	\sigma(x_{(1)},y_{(1)})\sigma(x_{(2)}y_{(2)},z)&=\sigma(y_{(1)},z_{(1)})\sigma(x,y_{(2)}z_{(2)}),\\\label{eqn:normalized}
	\sigma(x,1)&=\eps(x)=\sigma(1,x),
\end{align}
for all $x,y,z\in H$. The set of Hopf 2-cocycles for $H$ is denoted by $Z^2(H)$.

Given $\sigma,\sigma'\in Z^2(H)$, we say that $\sigma$ and $\sigma'$ are {\it cohomologous} and write $\sigma\sim\sigma'$ if there exists a convolution unit $\alpha\in U(H^\ast)$ such that $\sigma'=\alpha\rightharpoonup\sigma$, where
\begin{align}\label{eqn:action-on-cocycles}	
	\alpha\rightharpoonup\sigma(x,y)=\alpha(x_{(1)})\alpha(y_{(1)})\sigma(x_{(2)},y_{(2)})\alpha^{-1}(x_{(3)}y_{(3)}),\quad x,y\in H.
\end{align}

\subsubsection{Cocycle deformations}
Let $\sigma\in Z^2(H)$. The convolution $m_\sigma=\sigma\ast m\ast\sigma^{-1}$ defines a new associative multiplication on $H$, so that $H_\sigma\coloneqq(H,m_\sigma,\Delta)$ is again a Hopf algebra. We refer to $H_\sigma$ as a {\it cocycle deformation} of $H$.
On the other hand, $H_{(\sigma)}=(H,\sigma\ast m)$ is a cleft object for $H$, with coaction induced by the comultiplication $\Delta\colon H_{(\sigma)}\to H_{(\sigma)}\ot H$. In fact, every cleft object arises in this way: If $(E,\gamma)\in \Cleft(H)$, then the formula
\begin{align}\label{eqn:recuperation-formula}
	\sigma(x,y)=\gamma(x_{(1)})\gamma(y_{(1)})\gamma^{-1}(x_{(2)}y_{(2)}), \quad x,y\in H.
\end{align}
defines a Hopf 2-cocycle for $H$, such that $E\simeq H_{(\sigma)}$. Moreover, it follows that the left Hopf algebra introduced in \cite{Sc96} satisfies $L(H_{(\sigma)},H)\simeq H_\sigma$.

\subsubsection{Exponential cocycles}
If $H=\bigoplus_{n\geq 0}H_n$ is a graded Hopf algebra, we recall the well-known characterization of $\Z^2(H,\k)$, the set of Hochschild 2-cocycles in $H$ with trivial coefficients, that is, a linear map $\eta\colon H\ot H\to \k$ lies in $\Z^2(H,\k)$ iff it satisfies
\begin{align*}
	\eta(1,a)=\eta(a,1)=0;\qquad \eta(a,bc)=\eta(ab,c),\qquad a,c\in H^+, b\in H.
\end{align*}
Here $H^+\coloneqq\bigoplus_{n>0}H_n$.

On the other hand, the set of 2-coboundaries $\B^2(H,\k)\subset\Z^2(H,\k)$ is characterized by those $\beta\in \Z^2(H,\k)$ for which there exists a linear morphism $f_\beta\colon H\to \k$ such that $f_\beta(1)=\beta(1,1)$ and $\beta(a,b)=-f_\beta(ab)$, $a,b\in H^+$.

Now, if $\eta\in\Z^2(H,\k)$ is a Hochschild 2-cocycle, then under certain conditions we can define a Hopf 2-cocycle by the exponential $e^\eta=\sum_{k\geq 0}\frac{1}{k!}\eta^{\ast k}$; see \cite{GaM} for a detailed discussion.
A Hopf 2-cocycle is called {\it pure} when it is not cohomologous to $e^\eta$ for any $\eta\in\Z^2(H,\k)$.

\subsection{Nichols algebras and braided cocycles}\label{sec:braided-context}
Let $V\in \ydh$ be a Yetter-Drinfeld module. A braided graded Hopf algebra $\mathfrak{B}=\bigoplus_{n\geq0}\Bq_n\in\ydh$ is called a Nichols algebra for $V$ if $\k\simeq \Bq_0$, $V\simeq \Bq_1\in \ydh$, $P(\mathfrak{B})=\Bq_1$, and $\Bq$ is generated as an algebra by $\Bq_1$. The usual notation is $\Bq=\Bq(V)$. Note that, by definition, $\Bq(V)$ is a quotient of the tensor algebra $T(V)$. 
A pre-Nichols algebra $\hat{\Bq}\in\ydh$ for $V$ is an intermediate quotient $T(V)\twoheadrightarrow \hat{\Bq}\twoheadrightarrow \Bq(V)$. See \cite{AS} for more details.

For a Nichols algebra $\Bq \in \ydh$, we denote by $Z^2(\mathfrak{B})^H$ the set of (braided) Hopf 2-cocycles for $\Bq$  that are $H$-linear, that is, $\sigma\colon \Bq \ot \Bq \to \k \in Z^2(\mathfrak{B})^H$ iff \eqref{eqn:cocycle-condition} and \eqref{eqn:normalized} hold, and 
$$\sigma(h_{(1)}\cdot a, h_{(2)}\cdot b)=\eps(h)\sigma(a,b), \qquad \text{for all } a,b\in \Bq.$$
Similarly, $\Z^2(\Bq,\k)^H$ (resp. $\B^2(\Bq,\k)^H$) denotes the set of Hochschild 2-cocycles (resp. 2-coboundaries) for $\Bq$ that are $H$-linear. 

We write 
$\Cleft_H(\Bq)$ for the set of isoclasses of cleft objects for $\Bq$ in $H\mod$.

\begin{remark}\label{rem:primera-linea}
	Let $\{x_i\}_{i\in \I_\theta}$ be a basis of $\Bq_1$ contained in a homogeneous basis $\BB$ of $\Bq$. Then each $\sigma\in Z^2(\Bq)^H$ is completely determined by its values on the {\it first line}, namely, by the set $\{\sigma(x_i,b)\colon b\in \BB,\, i\in\I_\theta\}$. See \cite[Lemma 3.2]{GIS1}.
\end{remark}
\medbreak
Assume that $H$ is semisimple, $\Bq$ is finitely presented, and set $A=\Bq\#H$. The strategy in \cite{AAGMV} constructs a collection of Hopf algebras $(A_{\bs\lambda})_{\bs\lambda\in\bs\Lambda}$ from the definition of a family $(\mE_{\bs\lambda})_{\bs\lambda\in\bs\Lambda}\subseteq\Cleft_H(\Bq)$. The Hopf algebras $A_{\bs\lambda}$ are isomorphic to the left Schauenburg algebras $L(E_{\bs\lambda},A)$, where $E_{\bs\lambda}$ is the $A$-cleft object given by the bosonization $E_{\bs\lambda}=\mE_{\bs\lambda}\# H$. This implies that $A_{\bs\lambda}$ is a cocycle deformation of $A$; the explicit form of the cocycles is not provided. Moreover, since $E_{\bs\lambda}=\mE_{\bs\lambda}\#H$ and the corresponding section $\gamma\colon A\to E_{\bs\lambda}$ factorizes as $\gamma=\gamma_{\mE_{\bs\lambda}}\#\id_H$, then the associated cocycle $\sigma\in Z^2(A)$, see \eqref{eqn:recuperation-formula}, belongs to the subset $Z_0^2(A)$ of 2-cocycles $\sigma$ satisfying $\sigma_{H\ot H}=\eps$, and thus it factorizes in turn as
\begin{align}\label{eqn:factor}
	\sigma=\sigma_{\bs\lambda}\# \eps, \qquad \sigma_{\bs\lambda}\in Z^2(\Bq)^H.
\end{align} 
The equivalence is given by $\sigma_{\bs\lambda}=\sigma_{|\Bq\ot \Bq}$ and  $\sigma(x\#h,y\#k)=\sigma_{\bs\lambda}(x,h\cdot y)\eps(k)$; see \cite{AG} for an extensive treatment. Notice that in this configuration, $\sigma_{\bs\lambda}$ is the cocycle associated to the section $\gamma_{\bs\lambda}$ in $H\mod$.


On the other hand, the factorization \eqref{eqn:factor} holds in the context of Hochschild 2-cocycles. In particular, this is a consequence of a combination of \c{S}tefan's results in \cite{St}, most notably that $\Z^2(A,\k)\simeq \Z^2(\Bq,\k)^H$.

In this context, $\sigma=\sigma_{\Bq}\#\eps\in Z_0^2(A)$ is pure if and only if $\sigma_{\Bq}\in Z^2(\Bq)^H$ is pure. That is, $\sigma\sim e^\eta$ for some $\eta\in\Z^2(A,\k)$ if and only if there exists $\eta_{\Bq}\in\Z^2(\Bq,\k)^H$ such that $\sigma_{\Bq}\sim e^{\eta_{\Bq}}$; moreover, $\eta=\eta_{\Bq}\#\eps$. In particular, to study the purity of cocycles, we can restrict our attention to the subset $\Z_0^2(\Bq,\k)^H$ of Hochschild 2-cocycles $\eta\in\Z^2(\Bq,\k)^H$ such that $\eta(1,1)=0$. We refer to \cite{GIS1} for further details.

As a result of the above discussion, we can fully restrict our computations to the braided setting.


\subsubsection{Braided vector spaces of diagonal type}
A {\it YD-pair} for a Hopf algebra $H$ is a pair $(g,\chi)\in G(H)\times \Alg(H,\k)$ such that
$\chi(h)g=\chi(h_{(2)})h_{(1)}g\Ss(h_{(3)})$, for all $h\in H$; this implies that $g\in Z(G(H))$.

A finite-dimensional braided vector space $(V,c)$ is said to be of {\it diagonal type} if there exists a basis $\{x_1,\dots,x_\theta\}$ of $V$ such that $c$ is represented by a matrix $\bq=(q_{ij})$, namely, 
$c(x_i\ot x_j)=q_{ij}x_j\ot x_i$, $i,j\in\I_\theta$; in this case, we denote $c = c^{\qb}$. 
A {\it principal YD-realization} of $(V,c^\qb)$ over $H$ is a family of YD-pairs $\{(g_i,\chi_i): i\in\I_\theta\}$ for $H$ 
such that $V\in\ydh$ with action and coaction 
\begin{align*}
	h\cdot x_i=\chi_i(h)x_i,\qquad x_i\mapsto g_i\ot x_i,\qquad h\in H,\ i\in \I_\theta,
\end{align*}
and such that $c^\qb$ coincides with the braiding in $\ydh$. Thus $\chi_j(g_i)=q_{ij}$, for all $i,j\in\I_\theta$. In this situation, the Nichols algebra of $(V,c^\qb)\in\ydh$ is denoted by $\Bq(V)=\Bq_\qb$.
See also \cite{AA} for a complete description of finite-dimensional Nichols algebras of diagonal type.

\medbreak

Let $q\in\G'_N$ with $N\geq 2$ and $\theta\in \N$. A braided vector space of diagonal type $(V,c^\qb)$ is said to be of {\it Cartan type $A_\theta$} if $\dim V=\theta$ and the braiding matrix $\qb=(q_{ij})$ satifies (for some basis)
\begin{align*}
	q_{ii}=q, \qquad q_{ij}q_{ji}=\begin{cases}
		q^{-1},\, &\text{if }\, |i-j|=1,\\
		1,\,  &\text{if }\, |i-j|\geq 2,
	\end{cases}\qquad \text{for all } i,j\in \I_\theta.
\end{align*}
As the name suggests, this braiding belongs to the family of {Cartan type} braidings; see \cite{AS00} for the general definition.

\subsection{Lifting of type $A_2$}\label{sec:liftA2}
Let $q$ be a primitive $N$th root of unity in $\k$, $N\geq 3$. We consider a braiding $\qb=\begin{psmallmatrix}q&q_{12}\\q_{21}&q
\end{psmallmatrix}$ of Cartan type $A_2$, that is, $q_{12}q_{21}=q^{-1}$. 

Let $\{x_1,x_2\}$ be a $\k$-basis associated to this braiding, and let $\Bq_\qb$ be the Nichols algebra associated to $(V,c^\qb)$, with $V\coloneqq\k\{x_1,x_2\}$.
From \cite[Proposition 1.1]{AAG}, $\Bq_\qb$ is the algebra generated by $x_1,x_2$ with relations:
\begin{align}\label{eqn:relations-nichols}
	x_1^N=0,&  &x_2^N=0, & & x_{12}^N=0,& &x_{112}=0,  && x_{122}=0,
\end{align}
for $x_{12}\coloneqq x_1x_2 - q_{12}x_2x_1$, $x_{112}\coloneqq x_1x_{12}-qq_{12}x_{12}x_1$, and $x_{122}\coloneqq x_{12}x_2-qq_{12}x_2x_{12}$.
\begin{remark}
	In \cite{AAG}, the relation $x_{221}=0$ is used instead of $x_{122}=0$, where $x_{221}=x_2x_{21}-qq_{21}x_{21}x_2$ for $x_{21}\coloneqq x_2x_1-q_{21}x_1x_2$. However, a direct computation shows that $x_{221}=0$ if and only if $x_{122}=0$. We prefer the presentation \eqref{eqn:relations-nichols} because it avoids introducing additional notation.
\end{remark}

The relations $x_{112}=0=x_{122}$ are called {\it quantum Serre relations}, since they can be expressed by a braided (or $q$-deformed) adjoint bracket, and hence they are precisely the braided analogues of the classical Serre relations from Lie theory. For more details, see \cite{AAG}.


A linear PBW basis for this algebra is given by
\begin{align}\label{eqn:basis}
	\mathbb{B}_\bq=\{x_2^{n_2}x_{12}^{n_{12}}x_1^{n_1}\colon n_2, n_{12}, n_1\in\I_{N-1}^\circ\}.
\end{align}
The {\it distinguished} pre-Nichols algebra $\tilde{\Bq}_{\qb}$ is the quotient $\k\lg x_1,x_2|x_{112}=x_{122}=0\rg$. It admits an (infinite) PBW basis
$\tilde{\BB}_\bq=\{x_2^{n_2}x_{12}^{n_{12}}x_1^{n_1}: n_2, n_{12}, n_1\geq 0\}$.

We fix a semisimple Hopf algebra $H$ with a YD-realization $(g_i,\chi_i)_{i\in\I_2}$ of $\bq$.

	
	The cleft objects for $\Bq_{\bq}$ are parametrized by the family of algebras $\mE_{\bs\lambda}$, with $\bs\lambda=(\lambda_{1},\lambda_{2},\lambda_{12},\lambda_{112},\lambda_{122})\in\k^5$, generated by $y_1, y_2$ and the relations:
	\begin{align}\label{eqn:rels_cleft}
		y_1^N&=\lambda_1, & y_2^N&=\lambda_2,  & y_{12}^N&=\lambda_{12},
		&y_{112}&=\lambda_{112},&y_{122}&=\lambda_{122},
	\end{align}
	where $y_{12}\coloneqq y_1y_2 - q_{12}y_2y_1$, $y_{112}\coloneqq y_1y_{12}-qq_{12}y_{12}y_1$, and $y_{122}\coloneqq y_{12}y_2-qq_{12}y_2y_{12}$, and the parameters $\bs\lambda\in\k^5$ must satisfy:
	\begin{align}\label{eqn:restictions-lambda}
		\begin{split}	\lambda_i\neq 0 &\text{ only if } \chi_i^N=\eps, \ i\in\I_2;  \qquad \lambda_{12}\neq 0 \text{ only if } (\chi_{1}\chi_{2})^N=\eps;\\
			\lambda_{112}\neq 0 &\text{ only if } \chi_{1}^2\chi_{2}=\eps; \quad\qquad\quad \lambda_{122}\neq 0 \text{ only if } \chi_{1}\chi_{2}^2=\eps.
	\end{split}\end{align}
	
	On the other hand, the deformations of $\Bq_\qb$ over $H$ are obtained as left Schauenburg Hopf algebras
	$\mathfrak{u}_\qb(\bs\lambda)\coloneqq L(\mE_{\bs\lambda}\#H,\Bq_\qb\#H)$, for $\mE_{\bs\lambda}\in\Cleft_H(\Bq_\qb)$. Explicitly, $\mathfrak{u}_\qb(\bs\lambda)$ is the algebra generated by $a_1,a_2$ with relations
	\begin{align}\label{rel:A2-N-3-generic}
		\begin{split}
			a_i^N&=\lambda_i(1-g_i^N),\, i\in \I_2,\quad a_{112}=\lambda_{112}(1-g_{1}^2g_{2}), \quad a_{122}=\lambda_{122}(1-g_{1}g_{2}^2),\\
			a_{12}^N&=\lambda_{12}(1-g_{1}^Ng_{2}^N)-(1-q^{-1})^Nq_{21}^{\frac{N(N-1)}{2}}\lambda_1\lambda_2(1-g_1^N)g_2^N,
		\end{split}
	\end{align}
	for $a_{12}\coloneqq a_1a_2 - q_{12}a_2a_1$, $a_{112}\coloneqq a_1a_{12}-qq_{12}a_{12}a_1$ y $a_{122}\coloneqq a_{12}a_2-qq_{12}a_2a_{12}$, where the parameters $\bs\lambda\in\k^5$ satisfy \eqref{eqn:restictions-lambda} and the normalization conditions 
	\begin{align}\label{restricciones-normalizacion-raices}
		\begin{split}
			\lambda_i&=0, \text{ si }\,  g_i^N=1, \, i\in \I_2;\qquad\lambda_{12}=0, \text{ si }\,  g_1^Ng_{2}^N=1;\\
			\lambda_{112}&=0, \text{ si }\,  g_{1}^2g_{2}=1;\qquad \qquad\,\,
			\lambda_{122}=0, \text{ si }\,  g_{1}g_{2}^2=1.
		\end{split}
	\end{align}
	In particular, any lifting arises as a cocycle deformation of  $\Bq_\qb\#H$, see \cite[Theorem 1.8, Theorem 1.10]{AAG}.

	The following lemma allows us to separate the analysis and computation of the cocycles associated with the liftings.
\begin{lemma}\cite[Lemma 5.1]{AAG}\label{lem:AAG} Let $\qb$, $\Bq_{\qb}$ and $\mE_{\bs\lambda}$ be as above.
\begin{itemize}
\item If $N>3$, then $\lambda_{112}=\lambda_{122}=0$.
\item If $N=3$ and $q_{12}\neq q_{21}$, then $\lambda_{112}=\lambda_{122}=0$.
\end{itemize}
Hence,  $\lambda_{112}$ or $\lambda_{122}$ may be nonzero if and only if $\qb=\begin{psmallmatrix}
q&q\\q&q
\end{psmallmatrix}$, $q\in\mathbb{G}_3'$. Thus:
\begin{itemize}
	\item If $\qb\neq \begin{psmallmatrix}
	q&q\\q&q
\end{psmallmatrix}$, then $\mE_{\bs\lambda}$ is an algebra quotient of $\tilde{\Bq}_{\qb}$.\qed
\end{itemize}
\end{lemma}
From the viewpoint of the cleft objects $\mE_{\bs\lambda}$, or even of the parameters $\bs\lambda \in \k^5$, this result shows that we have two distinct cases for the computation of Hopf 2-cocycles: the {\it generic liftings}, where the quantum-Serre relations $a_{112}=a_{122}=0$ hold, and the {\it atypical liftings}, where such relations may be deformed (hence $\qb= \begin{psmallmatrix}
	q&q\\q&q
\end{psmallmatrix}$, $q\in\mathbb{G}_3'$), {\it cf.}~\eqref{rel:A2-N-3-generic}. We deal with these instances in Sections \ref{sec:general} and \ref{sec:N3}, respectively.

\section{A good setting}\label{sec:Good-Setting}
Consider a braided vector space $(V,c)$	with basis $\{x_1,\dots,x_\theta\}$. Suppose that $V$ admits a realization in $\ydh$, for some Hopf algebra $H$; and denote $\Bq=\Bq(V)$. 
Let $\mE\in\Cleft_H(\Bq)$, with coaction $\rho\colon\mE\to\mE\ot\Bq$.
We fix:
\begin{itemize}[leftmargin=*]
	\item A pre-Nichols algebra $\hat{\Bq}$ with an algebra projection $\pi_{\mE}\colon \hat{\Bq}\to \mE$ satisfying
\begin{align}\label{eqn:coaction-of-E}
	\rho\pi_\mathcal{E}=(\pi_\mathcal{E}\ot\pi_\Bq)\Delta_{\hat{\Bq}}.
	\end{align} 
Here, $\pi_{\Bq}\colon\hat{\Bq}\to\Bq$ denotes the canonical projection, and we set $\mathfrak{J}\coloneqq\ker\pi_{\Bq}$. 
\item A homogeneous basis $\BB$ of $\Bq$: that is, a collection $\{p_b(x_1,\dots,x_\theta):b\in\BB\}$, where each $p_b$ is a homogeneous (non-commutative) polynomial in  $\k\langle t_1,\dots,t_\theta\rangle$. 
\end{itemize}
\begin{example}\label{exa:basis}
Consider the Nichols algebra $\Bq_{\bq}$ of Cartan type $A_2$ as in \S\ref{sec:liftA2}, then the basis \eqref{eqn:basis} is determined by the non-commutative polynomials in two variables. Indeed, for each $b=x_2^{n_2}x_{12}^{n_{12}}x_1^{n_1}\in \BB_\qb$, we define $p_b\in \k\langle t_1,t_2\rangle$ as $p_b\coloneqq t_2^{n_2}(t_1t_2-q_{12}t_2t_1)^{n_{12}}t_1^{n_1}$. Thus, the specialization $t_i\mapsto x_i$, $i\in \I_2$, establishes that $b=p_b(x_1,x_2)$. 
\end{example}

\begin{remark}
Recall that such $\hat{\Bq}$ as above always exists, as we can take $\hat{\Bq}=T(V)$.
\end{remark}

Since $\mE$ is generated by $V$ as an algebra, we rename the generators as $y_1,\dots,y_\theta$.
Let $\EE=\{p_b(y_1,\dots,y_\theta):b\in\BB\}$ denote the corresponding basis of $\mE$ and let $\ell\colon\Bq\to\mE$ be the isomorphism such that $\ell(\BB)=\EE$. 
We consider the linear map
\[
\eps_0\colon\mE\to \k \quad \text{ so that } \, \eps_0\ell=\eps_{\Bq}.
\]
Under these assumptions, we obtain a simplification of \eqref{eqn:recuperation-formula}.
\begin{proposition}\label{pro:intro}
Let $\gamma\colon\Bq\to\mE$ be an $H$-linear section. If $\eps_0\gamma=\eps_\Bq$, then 
	\[
	\sigma(a,b)=\eps_0(\gamma(a)\gamma(b)),  \ a,b\in\Bq.
	\]
	is an $H$-invariant Hopf cocycle such that $\mE\simeq \Bq_{(\sigma)}$.
\end{proposition}
	\begin{proof}
By \cite[Theorem 1.3]{AF}, we have the braided analogous of \eqref{eqn:recuperation-formula}, namely  $\sigma(a,b)=\gamma(a_{(1)})\gamma(b_{(1)})\gamma^{-1}(a_{(2)}b_{(2)})$, 
$a,b\in\Bq$. Then, the formula is equivalent to 
\begin{align}\label{pro:intro-proof}
\sigma(a_{(1)},b_{(1)})\gamma(a_{(2)}b_{(2)})=\gamma(a)\gamma(b)
\end{align}
Now, using the braided coproduct in $\Bq\ot \Bq$, the equality \eqref{pro:intro-proof} is
\begin{multline*}
		\sigma(a,b)+\sigma(a,b_{\underline{(1)}})\gamma(b_{\underline{(2)}})+\sigma(a_{\underline{(1)}},a_{\underline{(2)}(-1)}\cdot b)\gamma(a_{\underline{(2)}(0)})\\+\sigma(a_{\underline{(1)}},a_{\underline{(2)}(-1)}\cdot b_{\underline{(1)}})\gamma(a_{\underline{(2)}(0)}b_{\underline{(2)}})+\gamma(ab)=\gamma(a)\gamma(b).
\end{multline*}
	As $\eps_0\gamma=\eps_\Bq$, we apply $\eps_0$ on both sides in this equality and obtain the result.
\end{proof}

We now write $\tx_1,\dots,\tx_\theta$ for the generators of $\hat{\Bq}$ and consider the inclusion $\iota\colon \Bq\to\hat{\Bq}$, determined by the liftings 
\begin{align}\label{eqn:iota}
b=p_b(x_1,\dots,x_\theta)\mapsto \tb\coloneqq p_b(\tx_1,\dots,\tx_\theta), \qquad b\in\BB.
\end{align}
 This establishes a decomposition of vector spaces $\hat{\Bq}=\iota(\Bq)\oplus \mathfrak{J}$; we let $\tau_1\colon \hat{\Bq}\to \iota(\Bq)$ and $\tau_2\colon \hat{\Bq}\to \mathfrak{J}$ be the corresponding linear projections. 
 
 Note that $\ell(\pi_{\Bq}(\tb))=\ell(b)=\pi_{\mE}(\tb)$, $b\in\BB$. We have the following result.

\begin{lemma}\label{lem:H1}
Assume that
\begin{align}\label{eqn:H1}\tag{{\bf H1}}
	((\eps_0\ot\ell)\rho\ot \id)(\pi_\mathcal{E}\tau_2\ot \pi_\Bq)\underline{\Delta}(\tb)&=(\pi_\mathcal{E}\tau_2\ot \pi_\Bq)\underline{\Delta}(\tb), &   \text{all }b&\in\mathbb{B}.
\end{align}
Then the following identity holds, for all $b\in\BB$:
	\begin{multline}\label{eqn:construction-coaction-equallity}
	(\pi_\mathcal{E}\tau_2\ot \pi_\Bq)\underline{\Delta}(\tb)+[((\eps_0\pi_\mathcal{E}\tau_2\ot \pi_\mathcal{E}\tau_1)\underline{\Delta}\tau_1)\ot \pi_\Bq]\underline{\Delta}(\tb)\\=(\eps_0\pi_\mathcal{E}\tau_2\ot\pi_\Bq )\underline{\Delta}(\tb)+[\eps_0\pi_\mathcal{E}\tau_2\ot ((\pi_\mathcal{E}\tau_1\ot \pi_\Bq)\underline{\Delta}\tau_1)]\underline{\Delta}(\tb).    
\end{multline}
\end{lemma}
Here, we consider $\eps_0\ot\ell$ as a map $\mE\ot\Bq\to \mE$, via the identification $\k\ot\mE\simeq \mE$.
\pf
The identity \eqref{eqn:construction-coaction-equallity} follows by first applying $\eps_0\ot \ell\ot\id$ on both sides of the identity $(\rho\ot \id)\rho(\pi_\mathcal{E}(\tb))=(\id\ot \Delta)\rho(\pi_\mathcal{E}(\tb))$, and using the hypothesis \eqref{eqn:H1}.

On the one hand,
	\begin{align}\begin{split}
\label{eqn:rho1}
	\rho\pi_\mathcal{E}(\tb)\stackrel{\eqref{eqn:coaction-of-E}}{=}&(\pi_\mathcal{E}\ot\pi_\Bq)\Delta(\tb)=\pi_\mathcal{E}(\tb)\ot 1+1\ot\pi_\Bq(\tb)+(\pi_\mathcal{E}\ot \pi_\Bq)\underline{\Delta}(\tb)\\ 
	=&\pi_\mathcal{E}(\tb)\ot 1+1\ot\pi_\Bq(\tb)+(\pi_\mathcal{E}\tau_2\ot \pi_\Bq)\underline{\Delta}(\tb)+(\pi_\mathcal{E}\tau_1\ot \pi_\Bq)\underline{\Delta}(\tb).
		\end{split}
	\end{align}
Next we apply $(\rho\ot\id)$, and obtain that $(\rho\ot \id)\rho\pi_\mathcal{E}(\tb)$ is
\begin{align*}
	\Big(\pi_\mathcal{E}(\tb)\ot 1+&1\ot\pi_\Bq(\tb)+(\pi_\mathcal{E}\tau_2\ot \pi_\Bq)\underline{\Delta}(\tb)+(\pi_\mathcal{E}\tau_1\ot \pi_\Bq)\underline{\Delta}(\tb)\Big)\ot 1\\
	&+1\ot1\ot \pi_\Bq(\tb)+(\rho\pi_\mathcal{E}\tau_2\ot \pi_\Bq)\underline{\Delta}(\tb)+(\rho\pi_\mathcal{E}\tau_1\ot \pi_\Bq)\underline{\Delta}(\tb).
\end{align*}
Let us apply $(\eps_0\ot \ell\ot \id)$: it is clear that
\[
(\eps_0\ot \ell\ot \id)(\pi_\mathcal{E}(\tb)\ot 1\ot 1+1\ot\pi_\Bq(\tb)\ot 1+1\ot1\ot \pi_\Bq(\tb))=\pi_\mathcal{E}(\tb)\ot 1+1\ot \pi_\Bq(\tb)
\]
and $(\eps_0\ot \ell\ot \id)(\pi_\mathcal{E}\tau_1\ot \pi_\Bq)\underline{\Delta}(\tb)\ot 1=0$. Hence, following \eqref{eqn:rho1} we get:
\[
(\eps_0\ot \ell\ot \id)(\rho\pi_\mathcal{E}\tau_1\ot \pi_\Bq)\underline{\Delta}(\tb)=(\pi_\mathcal{E}\tau_1\ot \pi_\Bq)\underline{\Delta}(\tb)
+[((\eps_0\pi_\mathcal{E}\tau_2\ot \pi_\mathcal{E}\tau_1)\underline{\Delta}\tau_1)\ot \pi_\Bq]\underline{\Delta}(\tb).    
\]
Hence, we reach the equality:
\begin{align}\label{eqn:leftside}
\notag	(\eps_0\ot \ell\ot\id)(\rho\ot \id)&\rho \pi_\mathcal{E}(\tb)=\pi_\mathcal{E}(\tb)\ot 1+1\ot \pi_\Bq(\tb)\\
	&+(\eps_0\ot \ell)(\pi_\mathcal{E}\tau_2\ot \pi_{\Bq})\underline{\Delta}(\tb)\ot 1+[((\eps_0\ot \ell)\rho\pi_\mathcal{E}\tau_2)\ot \pi_\Bq]\underline{\Delta}(\tb)\\
\notag	&+(\pi_\mathcal{E}\tau_1\ot \pi_\Bq)\underline{\Delta}(\tb)
	+[((\eps_0\pi_\mathcal{E}\tau_2\ot \ell\pi_{\Bq})\underline{\Delta}\tau_1)\ot \pi_\Bq]\underline{\Delta}(\tb).    
\end{align}

On the other hand, by \eqref{eqn:rho1} and the definition of $\underline{\Delta}$:
\begin{align*}
	(\id\ot \Delta)\rho\pi_\mathcal{E}(\tb)=\pi_\mathcal{E}(\tb)&\ot 1\ot 1+1\ot[\pi_\Bq(\tb)\ot 1+1\ot\pi_\Bq(\tb) +(\pi_\Bq\ot \pi_\Bq)\underline{\Delta}(\tb)]\\
	&+(\pi_\mathcal{E}\tau_2\ot \Delta\pi_\Bq)\underline{\Delta}(\tb)+[(\pi_\mathcal{E}\tau_1)\ot \Delta\pi_\Bq]\underline{\Delta}(\tb).  
\end{align*}
	Applying $(\eps_0\ot \ell\ot \id)$, the first line above becomes
\[
\pi_\mathcal{E}(\tb)\ot1+1\ot \pi_\Bq(\tb)+(\pi_\mathcal{E}\tau_1\ot \pi_\Bq)\underline{\Delta}(\tb)
\]
while $(\eps_0\ot \ell\ot \id)[(\pi_\mathcal{E}\tau_1)\ot \Delta\pi_\Bq]\underline{\Delta}(\tb)=0$. As well:
\begin{align*}
(\eps_0\ot \ell\ot \id)&(\pi_\mathcal{E}\tau_2\ot \Delta\pi_\Bq)\underline{\Delta}(\tb)=(\eps_0\pi_\mathcal{E}\tau_2\ot \ell\pi_{\Bq})\underline{\Delta}(\tb)\ot 1\\
\notag &+(\eps_0\pi_\mathcal{E}\tau_2\ot\pi_\Bq )\underline{\Delta}(\tb)	+
(\eps_0\ot \ell\ot \id)(\pi_\mathcal{E}\tau_2\ot\pi_{\Bq}\ot \pi_\Bq)(\id\ot \underline{\Delta}\tau_1)\underline{\Delta}(\tb),
\end{align*}
using that $\Delta\pi_{\Bq}=\pi_{\Bq}\ot 1+1\ot\pi_{\Bq}+(\pi_{\Bq}\ot\pi_{\Bq})\underline{\Delta}\tau_1$. Hence
\begin{align}\label{eqn:rightside}
\notag (\eps_0\ot &\ell\ot \id)(\id\ot \Delta)\rho\pi_\mathcal{E}(\tb)=\pi_\mathcal{E}(\tb)\ot1+1\ot \pi_\Bq(\tb)+(\pi_\mathcal{E}\tau_1\ot \pi_\Bq)\underline{\Delta}(\tb)\\
 &+(\eps_0\ot\ell)(\pi_\mathcal{E}\tau_2\ot \pi_{\Bq})\underline{\Delta}(\tb)\ot 1+(\eps_0\pi_\mathcal{E}\tau_2\ot\pi_\Bq )\underline{\Delta}(\tb)\\
\notag	&+(\eps_0\ot \ell\ot \id)(\pi_\mathcal{E}\tau_2\ot\pi_{\Bq}\ot \pi_\Bq)(\id\ot \underline{\Delta}\tau_1)\underline{\Delta}(\tb).
\end{align}

Finally, comparing \eqref{eqn:leftside} and \eqref{eqn:rightside}, we get:
\begin{align*}
[((\eps_0\ot \ell)&\rho\pi_\mathcal{E}\tau_2)\ot \pi_\Bq]\underline{\Delta}(\tb)
	+[((\eps_0\pi_\mathcal{E}\tau_2\ot \ell\pi_{\Bq})\underline{\Delta}\tau_1)\ot \pi_\Bq]\underline{\Delta}(\tb)
	\\
	&=(\eps_0\pi_\mathcal{E}\tau_2\ot\pi_\Bq )\underline{\Delta}(\tb)+(\eps_0\ot \ell\ot \id)(\pi_\mathcal{E}\tau_2\ot\pi_{\Bq}\ot \pi_\Bq)(\id\ot \underline{\Delta}\tau_1)\underline{\Delta}(\tb).
\end{align*}
By \eqref{eqn:H1}, the first summand in the left hand side becomes $(\pi_\mathcal{E}\tau_2\ot \pi_\Bq)\underline{\Delta}(\tb)$. Hence  we obtain \eqref{eqn:construction-coaction-equallity}.
\epf

\begin{lemma}\label{lem:H2}
Assume that \eqref{eqn:H1} holds and 
\begin{align}\label{eqn:H2}\tag{{\bf H2}}			[\eps_0\pi_\mathcal{E}\tau_2\ot ((\pi_\mathcal{E}\tau_2\ot \pi_\Bq)\underline{\Delta}\tau_1)]\underline{\Delta}(\tb)&=0,  \quad \text{all }b\in\mathbb{B}.
\end{align}
Then the linear map $\omega\colon \Bq\to\mE$ defined in the basis $\mathbb{B}$ via:
	\begin{align}\label{eqn:general-section}
	\omega(b)\coloneqq \ell(b)-(\eps_0\ot\ell)(\pi_\mathcal{E}\tau_2\ot \pi_\Bq)\underline{\Delta}(\tb)
\end{align}
is a convolution invertible comodule map.
\end{lemma}
\pf
We have $\rho(\omega(b))=\rho\ell(b)-(\eps_0\ot\rho\ell)(\pi_\mathcal{E}\tau_2\ot \pi_\Bq)\underline{\Delta}(\tb)$. Recall that we have already described $\rho\ell(b)=\rho(\pi_{\mE}(\tb))$ in \eqref{eqn:rho1}. Now, since  $\id_{\hat{\Bq}}=\tau_1+\tau_2$:
\begin{align*}
(\eps_0&\ot\rho\ell)(\pi_\mathcal{E}\tau_2\ot \pi_\Bq)\underline{\Delta}(\tb)=	(\eps_0\pi_\mathcal{E}\tau_2\ot \rho\pi_\mathcal{E}\tau_1)\underline{\Delta}(\tb)
\\
&\stackrel{\eqref{eqn:coaction-of-E}}{=}(\eps_0\pi_\mathcal{E}\tau_2\ot (\pi_\mathcal{E}\ot\pi_{\Bq})\Delta\tau_1)\underline{\Delta}(\tb)
=(\eps_0\pi_\mathcal{E}\tau_2\ot \pi_\mathcal{E}\tau_1)\underline{\Delta}(\tb)\ot 1\\
&\qquad \qquad +(\eps_0\pi_\mathcal{E}\tau_2\ot \pi_\Bq)\underline{\Delta}(\tb)+\Big(\eps_0\pi_\mathcal{E}\tau_2\ot ((\pi_\mathcal{E}(\tau_1+\tau_2)\ot \pi_\Bq)\underline{\Delta}\tau_1)\Big)\underline{\Delta}(\tb)\\
&\stackrel{\eqref{eqn:H2}}{=}(\eps_0\pi_\mathcal{E}\tau_2\ot \pi_\mathcal{E}\tau_1)\underline{\Delta}(\tb)\ot 1+(\eps_0\pi_\mathcal{E}\tau_2\ot \pi_\Bq)\underline{\Delta}(\tb)\\
&\qquad \qquad +\Big(\eps_0\pi_\mathcal{E}\tau_2\ot ((\pi_\mathcal{E}\tau_1\ot \pi_\Bq)\underline{\Delta}\tau_1)\Big)\underline{\Delta}(\tb).
\end{align*}
Observe that 
$(\eps_0\pi_\mathcal{E}\tau_2\ot \pi_\mathcal{E}\tau_1)\underline{\Delta}(\tb)\ot 1=(\eps_0\pi_\mathcal{E}\tau_2\ot \ell\pi_{\Bq})\underline{\Delta}(\tb)\ot 1$, 
hence
\begin{align*}
	\rho&(\omega(b))=\omega(b)\ot 1+1\ot b+(\pi_\mathcal{E}\tau_2\ot \pi_\Bq)\underline{\Delta}(\tb)
	+(\pi_\mathcal{E}\tau_1\ot \pi_\Bq)\underline{\Delta}(\tb)\\
	&\qquad -(\eps_0\pi_\mathcal{E}\tau_2\ot \pi_\Bq)\underline{\Delta}(\tb)-\Big(\eps_0\pi_\mathcal{E}\tau_2\ot ((\pi_\mathcal{E}\tau_1\ot \pi_\Bq)\underline{\Delta}\tau_1)\Big)\underline{\Delta}(\tb)\\
	&\stackrel{\eqref{eqn:construction-coaction-equallity}}{=}\omega(b)\ot 1+1\ot b+
	(\pi_\mathcal{E}\tau_1\ot \pi_\Bq)\underline{\Delta}(\tb) - 
	[((\eps_0\pi_\mathcal{E}\tau_2\ot \pi_\mathcal{E}\tau_1)\underline{\Delta}\tau_1)\ot \pi_\Bq]\underline{\Delta}(\tb)\\
	&=\omega(b)\ot 1+1\ot b+(\omega\ot\id)(\pi_\Bq\ot \pi_\Bq)\underline{\Delta}(\tb)
	=(\omega\ot\id)\Delta(b).
\end{align*}
Therefore,  $\rho\omega(b)=(\omega\ot \id)\Delta(b)$, for all $b\in\BB$, and thus it is a comodule map. 

Finally $\omega$ is convolution invertible, as $\omega(1)=1$.
\epf

\begin{theorem}\label{thm:H12}
	Let $\Bq$ and $\mathcal{E}\in\Cleft_H(\Bq)$ be as before. Assume that \textbf{H1} and \textbf{H2} hold. 
		If $\omega$ as in \eqref{eqn:general-section} is $H$-linear, then $\sigma=\eps_0(\omega\ot\omega)\in Z^2(\Bq)^H$ and $\mE\simeq \Bq_{(\sigma)}$.
\end{theorem}
\pf
By Lemma \ref{lem:H2}, $\omega\colon \Bq\to\mE$ is a convolution-invertible comodule map with $\omega(1)=1$, that is, it is a section. 
The result follows by Proposition \ref{pro:intro} as  $\omega$ verifies $\eps_0\omega=\eps_{\Bq}$ by definition.
\epf

\begin{remark}
Conditions \eqref{eqn:H1} and \eqref{eqn:H2} are imposed to reduce the computational complexity in the case under study-namely, the computation of the Hopf 2-cocycles of type $A_2$ in the next Sections \ref{sec:general} and \ref{sec:N3}. However, we have found examples of diagonal type in which these hypotheses do not hold. 
\end{remark}


\section{Hopf algebras of Cartan type $A_2$: generic liftings}\label{sec:general}

We focus on deformations of a Nichols algebra of Cartan type $A_2$, with $q\in\G'_N$ and $N>2$, as in \eqref{eqn:relations-nichols}. 
We fix a semisimple Hopf algebra $H$ such that $\Bq_\bq\in\ydh$ via a YD-realization $(g_i,\chi_i)_{i\in\I_2}$.

Recall from Lemma \ref{lem:AAG} that we have two distinct cases, according to the possibility of deforming the quantum Serre relations.
In this section we compute the Hopf 2-cocycles involved in the liftings of $\Bq_\qb$ under the assumption that the relations $x_{112}=x_{122}=0$ hold. 
In particular, we get that $y_{112}=y_{122}=0$ in any cleft object $\mE_{\bs\lambda}$ as well; see \eqref{eqn:rels_cleft}.

We work in the setting of \eqref{eqn:coaction-of-E}, where $\hat{\Bq}=\tilde{\Bq}$ denotes the distinguished pre-Nichols algebra in \S\ref{sec:liftA2}.

\subsection{The coproduct}\label{sec:copro}
We begin by describing the coproduct $\Delta\colon \tilde{\Bq}\to \tilde{\Bq}\ot \tilde{\Bq}$. 
Consider the following elements in $\tilde{\Bq}\ot \tilde{\Bq}$:
\begin{align}\label{eqn:notation-for-coproduct}
	X_i\coloneqq&x_i\ot 1, &Z_i\coloneqq&1\ot x_i, &X_{12}\coloneqq&x_{12}\ot 1;& Z_{12}\coloneqq&1\ot x_{12}, & i\in\I_2.
\end{align} 
Note that $\Delta(x_i)=X_i+Z_i$ and $\Delta(x_{12})=X_{12}+(1-q^{-1})X_1Z_2+Z_{12}$.

For integers $n_2, n_{12}, n_1\in \mathbb{Z}_{\geq 0}$, we fix the set of indexes:
$$P(n_2, n_{12}, n_1)\coloneqq\{(j,k,l,m)\in \mathbb{Z}^4\colon0\leq j\leq n_2,\, 0\leq l\leq k\leq  n_{12},\, 0\leq m\leq n_1\}$$ 
and $\underline{P}(n_2, n_{12}, n_1)\coloneqq P(n_2, n_{12}, n_1)\setminus\{(0,0,0,0),(n_2,n_{12},n_{12},n_1)\}$.

If $(j,k,l,m)\in P(n_2, n_{12}, n_1)$, then we set:
\begin{align*}
	&B_{(j,k,l,m)}\coloneqq \binom{n_2}{j}_{q}\binom{n_{12}}{k}_{q}\binom{k}{l}_{ q}\binom{n_1}{m}_{q},\\
	&Q_{(j,k,l,m)}\coloneqq (1-q^{-1})^{k-l}q_{21}^{(n_2-j)(k+m)+m(n_{12}-l)+\frac{(k-l-1)(k-l)}{2}}q^{(n_2-j)l+m(n_{12}-k)}.
\end{align*}
We also set $C_{(j,k,l,m)}\coloneqq B_{(j,k,l,m)}Q_{(j,k,l,m)}$ and write $C_{(k,m)}=C_{(0,k,0,m)}$.

\begin{lemma}\label{lem:coproduct in pre-Nichols}
	Let $n_2, n_{12}, n_1\in \mathbb{Z}_{\geq 0}$, then
	\begin{align}\label{eqn:coproduct-preNichols}
		\Delta(x_2^{n_2}x_{12}^{n_{12}}x_1^{n_1})=\sum_{\mathclap{P(n_2, n_{12}, n_1)}}C_{(j,k,l,m)}x_2^jx_{12}^lx_1^{k+m-l}\ot x_2^{n_2+k-j-l}x_{12}^{n_{12}-k}x_1^{n_1-m}.
	\end{align}
\end{lemma}
\pf
We first observe that $\Delta(x_i^{n})=(Z_i+X_i)^n$ and that $Z_iX_i=qX_iZ_i$. Hence, 
\begin{align*}
	\Delta(x_i^n)=\sum_{k=0}^n\binom{n}{k}_qX_i^kZ_i^{n-k}=\sum_{k=0}^n\binom{n}{k}_qx_i^k\ot x_i^{n-k}, \ i\in\I_2.
\end{align*}
On the other hand, as $Z_{12}(X_{12}+(1-q^{-1})X_1Z_2)=q(X_{12}+(1-q^{-1})X_1Z_2)Z_{12}$ and $X_1Z_2X_{12}=qX_{12}X_1Z_2$ we obtain:
\begin{align*}
	((1-q^{-1})X_1Z_2+X_{12})^k
	=&\sum_{l=0}^k\binom{k}{l}_q(1-q^{-1})^{k-l}q_{21}^{\frac{(k-l-1)(k-l)}{2}}X_{12}^lX_1^{k-l}Z_2^{k-l}.
\end{align*}
Therefore, this gives:
\begin{align*}
	\Delta(x_{12}^{n_{12}})=\Delta(x_{12})^{n_{12}}
	=\sum_{k=0}^{n_{12}}\sum_{l=0}^k\binom{n_{12}}{k}_q\binom{k}{l}_q(1-q^{-1})^{k-l}x_{12}^lx_1^{k-l}\ot x_2^{k-l}x_{12}^{n_{12}-k}.
\end{align*}
Finally, combining the expressions for the coproducts and the braiding relations, the formula follows.
\epf
In particular, for each $n_2, n_{12}, n_1\in\mathbb{Z}_{\geq 0}$ we obtain: 
\begin{align*}
	\underline{\Delta}(x_2^{n_2}x_{12}^{n_{12}}x_1^{n_1})=\sum_{\mathclap{\underline{P}(n_2, n_{12}, n_1)}}C(j,k,l,m)x_2^jx_{12}^lx_1^{k+m-l}\ot x_2^{n_2+k-j-l}x_{12}^{n_{12}-k}x_1^{n_1-m}.
\end{align*}

\begin{remark}\label{rem:coproduct-in-Bq}
	Formula \eqref{eqn:coproduct-preNichols} also applies to the coproduct $\Delta:\Bq_\bq\to\Bq_\bq\ot \Bq_\bq$, provided the exponents do not exceed $N - 1$.
\end{remark}


\subsection{The section} 

We now verify that Hypotheses \textbf{H1} and \textbf{H2} are satisfied. Consequently, we obtain a section $\gamma_{\bs\lambda}$ as in \eqref{eqn:general-section}, using Theorem \ref{thm:H12}.
		
		\begin{lemma}\label{lem:B-verifies-H1-H2}
The Hypotheses \textbf{H1} and \textbf{H2} hold.
		\end{lemma}
		\pf
		Fix $n_2, n_{12}, n_1\in \I_{N-1}^\circ$. On the one hand, when $n_{12}+n_1<N$, \eqref{eqn:coproduct-preNichols} gives $(\tau_2\ot\id) \underline{\Delta}(x_2^{n_2}x_{12}^{n_{12}}x_1^{n_1})=0$; hence it follows from \eqref{eqn:coaction-of-E} that
		\begin{align}\label{eqn:ast}
(\eps_0\ot \ell)\rho(y_2^{n_2}y_{12}^{n_{12}}y_1^{n_1})=1\ot y_2^{n_2}y_{12}^{n_{12}}y_1^{n_1}.
		\end{align}
		Assume now that $n_{12}+n_1\geq N$. As $n_{12},n_1<N$, then $N\leq n_{12}+n_1<2N$. 
		Set $$S=S(n_2, n_{12}, n_1)\coloneqq\{(j,k,l,m)\in \underline{P}(n_2, n_{12}, n_1)\colon k+m-l\geq N\}.$$
	Note that $(j,k,l,m)\in S$ implies $k+m-N<N$.	Then we get
		\begin{align*}
			&(((\eps_0\ot\ell)\rho\pi_\mathcal{E}\tau_2)\ot \pi_\Bq)\underline{\Delta}(x_2^{n_2}x_{12}^{n_{12}}x_1^{n_1})
			\\&=\lambda_1(((\eps_0\ot\ell)\rho)\ot \pi_\Bq)(\sum_{s\in S}C_{s}\,y_2^jy_{12}^ly_1^{k+m-N-l}\ot x_2^{n_2+k-j-l}x_{12}^{n_{12}-k}x_1^{n_1-m})
			\\&\stackrel{\eqref{eqn:ast}}{=}\lambda_1(\id\ot \pi_\Bq)(\sum_{s\in S}C_{s}\, y_2^jy_{12}^ly_1^{k+m-N-l}\ot x_2^{n_2+k-j-l}x_{12}^{n_{12}-k}x_1^{n_1-m})
			\\&=(\pi_\mathcal{E}\tau_2\ot \pi_\Bq)(\sum_{s\in \underline{P}}C_{s}\, x_2^jx_{12}^lx_1^{k+m-l}\ot x_2^{n_2+k-j-l}x_{12}^{n_{12}-k}x_1^{n_1-m})
			\\&=(\pi_\mathcal{E}\tau_2\ot \pi_\Bq)\underline{\Delta}(x_2^{n_2}x_{12}^{n_{12}}x_1^{n_1}).
		\end{align*}
		Therefore \textbf{H1} holds.
				To check \textbf{H2}, we observe that for $n_2, n_{12}, n_1<N$, we have:
		$$(\tau_2\ot \id)\underline{\Delta}(x_2^{n_2}x_{12}^{n_{12}}x_1^{n_1})=\sum_{\mathclap{S(n_2, n_{12}, n_1)}}C_{(j,k,l,m)}x_2^jx_{12}^lx_1^{k+m-l}\ot x_2^{n_2+k-j-l}x_{12}^{n_{12}-k}x_1^{n_1-m}.$$
		Thus in any case, as $n_{12}-k+n_1-m<N$, it follows $[\tau_2\ot ((\tau_2\ot \id)\underline{\Delta}\tau_1)]\underline{\Delta}=0$.
		\epf
		
		\begin{corollary}\label{cor:gamma-formula-QS-no-deformed}
			The linear map $\gamma_{\bs\lambda}\colon\Bq_\bq\to \mathcal{E}_{\bs\lambda}$ given in the basis $\mathbb{B}_\qb$ by
			\begin{align*}
				\gamma_{\bs\lambda}(x_2^{n_2}x_{12}^{n_{12}}x_1^{n_1})=\begin{cases}
					y_2^{n_2}y_{12}^{n_{12}}y_1^{n_1},  & n_{12}+n_1<N,\\
					y_2^{n_2}y_{12}^{n_{12}}y_1^{n_1}-\lambda_1\sum\limits_{\mathclap{\substack{\underline{P}(n_2, n_{12}, n_1)\\k+m=N\\n_2+k<N}}}C_{(k,m)} y_2^{n_2+k}y_{12}^{n_{12}-k}y_1^{n_1-m},  & n_{12}+n_1\geq N.
				\end{cases}
			\end{align*}
			is a section in $H\mod$ such that $\eps_0\gamma_{\bs\lambda}=\eps_{\Bq_\bq}$.
		\end{corollary}
		\pf
		Observe that the formula for $\gamma_{\bs\lambda}$ is obtained as in \eqref{eqn:general-section} for this context; hence, it verifies $\eps_0\gamma_{\bs\lambda}=\eps_{\Bq_\bq}$. The fact that $\gamma_{\bs\lambda}$ is a comodule map follows from Theorem \ref{thm:H12} and Lemma \ref{lem:B-verifies-H1-H2}. Notice that if $\lambda_1\neq 0$ and $k+m=N$, then it is $\chi_2^{n_2+k+n_{12}-k}\chi_1^{n_{12}-k+n_1-m}=\chi_2^{n_2+n_{12}}\chi_1^{n_{12}+n_1}$. Hence, $\gamma_{\bs\lambda}\in H\mod$.
		\epf
		
		\subsection{Explicit computations in $\mE_{\bs\lambda}$}
		
		The computation of the 2-cocycle $\sigma_{\bs\lambda}$ requires expressing the product of two basis elements, $y_2^{n_2}y_{12}^{n_{12}}y_1^{n_1}$ and $y_2^{m_2}y_{12}^{m_{12}}y_1^{m_1}$, as a linear combination 
		$\sum_{k_2, k_{12}, k_1}\alpha_{(k_2,k_{12},k_1)}y_2^{k_2}y_{12}^{k_{12}}y_1^{k_1}$
		in the same basis. More precisely, we are interested in the component lying in the trivial stratum $\mE_0$ of the filtration; namely, the coordinate $\alpha_{(0,0,0)}=\eps_0(y_2^{n_2}y_{12}^{n_{12}}y_1^{n_1}\cdot y_2^{m_2}y_{12}^{m_{12}}y_1^{m_1})$.
		
We fully describe this product in Proposition \ref{pro:product}. 
We isolate a key step in a previous lemma. Fix, for $n_1,m_2\in \I^\circ_{N-1}$ and $0\leq s\leq \min\{n_1,m_2\}$:
\[
L_{(n_1,s,m_2)}=\binom{n_1}{s}_q\binom{m_2}{s}_q(s)!_qq_{12}^{m_2n_1-\frac{s(s+1)}{2}}.
\]
		
%
		\begin{lemma}\label{lem:x1-x2}
			Let $n_1,m_2\in\mathbb{Z}_{\geq 0}$. The following identity holds in $\mE_{\bs\lambda}$:
			\begin{align}\label{eqn:x1-over-x2-c-A}
				y_1^{n_1}y_2^{m_2}=\sum_{s=0}^{\mathclap{\min\{n_1,m_2\}}}L_{(n_1,s,m_2)}y_2^{m_2-s}y_{12}^sy_1^{n_1-s}.
			\end{align}
		\end{lemma}
		\pf
				We assume $m_2\leq n_1$; the other case is analogous. The proof proceeds by induction and makes use of the right identity in \eqref{eqn:binomial-gaussian-identities-1}, valid for $1 \leq k \leq n$, namely:
		\begin{align}\label{eqn:binomial-gaussian-identities}
			&{q}^k\binom{n}{k}_{q}+\binom{n}{k-1}_{q}=\binom{n+1}{k}_{q}.
		\end{align}	
		
First, it is easy to use induction to check that:
		\begin{align}\label{eqn:claim1}
			y_1^ny_2=q_{12}^{n}y_2y_1^n+q_{12}^{n-1}(n)_q y_{12}y_1^{n-1}, \quad n\in\I_N.
		\end{align}
Then, we proceed by induction on 	$0\leq m_2<n_1$: with the base case $m_2 = 0$ which is obvious. We have:
		\begin{align*}
			y_1^{n_1}&y_2^{m_2+1}=\sum_{0\leq s\leq m_2}L_{(n_1,m_2,s)}y_2^{m_2-s}y_{12}^sy_1^{n_1-s}y_2\\
			\stackrel{\eqref{eqn:claim1}}{=}&\sum_{0\leq s\leq m_2}L_{(n_1,m_2,s)}y_2^{m_2-s}y_{12}^s(q_{12}^{n_1-s}y_2y_1^{n_1-s}+q_{12}^{n_1-s-1}(n_1-s)_q y_{12}y_1^{n_1-s-1})\\
			\stackrel{\eqref{eqn:binomial-gaussian-identities}}{=}&\sum_{0< s\leq m_2}\left(\binom{m_2+1}{s}_q-\binom{m_2}{s-1}_q\right)\tfrac{(n_1)!_q}{(n_1-s)_q!}q_{12}^{(m_2+1)n_1-\tfrac{s(s+1)}{2}}y_2^{m_2+1-s}y_{12}^sy_1^{n_1-s}\\
			+&\sum_{0\leq s< m_2}\binom{m_2}{s}_q\tfrac{(n_1)!_q}{(n_1-s-1)!_q}q_{12}^{(m_2+1)n_1-\tfrac{s(s+1)}{2}-s-1}y_2^{m_2-s}y_{12}^{s+1}y_1^{n_1-(s+1)}\\
			+&q_{12}^{(m_2+1)n_1}y_2^{m_2+1}y_1^{n_1}+\tfrac{(n_1)!_q}{(n_1-(m_2+1))!_q}q_{12}^{(m_2+1)n_1-\tfrac{(m_2+2)(m_2+1)}{2}}y_{12}^{m_2+1}y_1^{n_1-(m_2+1)}\\
			\\
			=&\sum_{0\leq s\leq m_2+1}\binom{m_2+1}{s}_q\tfrac{(n_1)!_q}{(n_1-s)!_q}q_{12}^{(m_2+1)n_1-\tfrac{s(s+1)}{2}}y_2^{m_2+1-s}y_{12}^sy_1^{n_1-s},
		\end{align*}
and this is $\sum\limits_{\mathclap{0\leq s\leq m_2+1}}L_{(n_1,m_2+1,s)}y_2^{m_2+1-s}y_{12}^sy_1^{n_1-s}$; the lemma follows.
		\epf
We compute the product of two basic elements, and its component $\eps_0$ on $\k$.
\begin{proposition}\label{pro:product}
The product of $y_2^{n_2}y_{12}^{n_{12}}y_1^{n_1}$ and $y_2^{m_2}y_{12}^{m_{12}}y_1^{m_1}$ in $\mE_{\bs\lambda}$ is given by:
\begin{align*}
\sum_{s=0}^{\mathclap{{\min\{n_1,m_2\}}}}L_{(n_1,s,m_2)}(q_{12}q)^{n_{12}(m_2-s)+(n_1-s)m_{12}}y_2^{n_2+m_2-s}y_{12}^{n_{12}+m_{12}+s}y_1^{n_1+m_1-s}.
\end{align*}
\end{proposition}
\pf
Follows by Lemma \ref{lem:x1-x2}; recall $y_{12}y_2=qq_{12}y_2y_{12}$ and $y_1y_{12}=qq_{12}y_{12}y_1$.
\epf

\begin{corollary}\label{cor:eps0}
Fix $n=(n_2,n_{12},n_1)$, $m=(m_2,m_{12},m_1)\in\I_{N-1}^{\circ \, 3}$ and set  $a\coloneqq n_2+m_2$, $b\coloneqq n_{12}+m_{12}$, $c\coloneqq n_1+m_1$.
Let $y_n=y_2^{n_2}y_{12}^{n_{12}}y_1^{n_1}$, $y_m=y_2^{m_2}y_{12}^{m_{12}}y_1^{m_1}\in\mE_{\bs\lambda}$.

Then $\eps_0(y_ny_m)\neq 0$ only if one of the following constraints $\mathfrak{c}_1, \dots, \mathfrak{c}_5$  holds:
\begin{itemize}
\item[$\mathfrak{c}_1)$] $a,b,c\in\{0,N\}$,
\item[$\mathfrak{c}_2)$] 	$n_2=m_1=0, n_1=m_2>0, b+n_1\in \{N,2N\}$,
\item[$\mathfrak{c}_3)$] 	$m_1=0, 0<n_1<m_2, a=N+n_1, b+n_1\in\{N,2N\}$,
\item[$\mathfrak{c}_4)$] $n_2=0,0<m_2<n_1, c=N+m_2, b+m_2\in\{N,2N\}$,
\item[$\mathfrak{c}_5)$] 	$a=c>N,a+b=2N,3N$.
\end{itemize}
The corresponding value is given by the table:
\begin{center}
\begin{tabular}{|c|c|}
	\hline
	\text{Constraint} & $\eps_0(y_{2}^{n_2}y_{12}^{n_{12}}y_{1}^{n_1}\cdot y_{2}^{m_2}y_{12}^{m_{12}}y_{1}^{m_1})$  \\
	\hline
$\mathfrak{c}_1$
			& $q_{12}^{n_1m_2}(qq_{12})^{n_{12}m_2+n_1m_{12}}\lambda_2^{a/N}\lambda_{12}^{b/N}\lambda_1^{c/N}$ \\
			\hline
$\mathfrak{c}_2$				& 
				$(n_1)!_qq_{12}^{\tfrac{n_1^2-n_1}{2}}\lambda_{12}^{{(b+n_1)}/{N}}$ \\
				\hline
$\mathfrak{c}_3$ & $\tfrac{(m_2)!_q}{(m_2-n_1)!_q}q_{12}^{\tfrac{2n_1m_2-(n_1^2+n_1)}{2}}(qq_{12})
				^{n_{12}(m_2-n_1)}\lambda_2\lambda_{12}^{{(b+n_1)}/{N}}$ \\
				\hline
				$\mathfrak{c}_4$ & $\tfrac{(n_1)!_q}{(n_1-m_2)!_q}q_{12}^{\tfrac{2n_1m_2-(m_2^2+m_2)}{2}}(qq_{12})
				^{m_{12}(n_1-m_2)}\lambda_{12}^{{(b+m_2)}/{N}}\lambda_1$ 
				\\
				\hline
			$\mathfrak{c}_5$ & 
				$L(n_1,a-N,m_1)(qq_{12})^{n_2n_{12}+m_{12}m_1}\lambda_2\lambda_{12}^{{(a+b-N)}/{N}}\lambda_1$
				\\\hline
			\end{tabular}
\end{center}
\end{corollary}
\pf
We see from Proposition \ref{pro:product} that $\eps_0(y_2^{n_2}y_{12}^{n_{12}}y_1^{n_1}\cdot y_2^{m_2}y_{12}^{m_{12}}y_1^{m_1})\neq 0$  only if $n_{12}+m_{12}+s\in\{0,N,2N\}$ and $n_i+m_i-s\in\{0,N\}$, $i\in\I_2$. The statement follows by analyzing each possible case.
\epf
%

\subsection{The cocycle}
By Theorem \ref{thm:H12} and Proposition \ref{pro:product}, we obtain the description of the 2-cocycle. Recall the definition 
of $\gamma_{\bs\lambda}\colon\Bq_\bq\to \mathcal{E}_{\bs\lambda}$ from Corollary \ref{cor:gamma-formula-QS-no-deformed}.

\begin{theorem}\label{thm:cocycle-generic}
Let $A$	be a lifting of $\Bq_{\qb}$ over a semisimple Hopf algebra $H$. Assume that the quantum Serre relations hold on $A$.
Then there exists $\bs\lambda=(\lambda_1,\lambda_2,\lambda_{12})\in\k^3$ such that $A\simeq (\Bq_{\qb}\#H)_{\sigma}$, with $\sigma=\sigma_{\bs\lambda}\#\eps$ and 
\begin{align}\label{eqn:cocycle-formula}
\sigma_{\bs\lambda}(x_2^{n_2}x_{12}^{n_{12}}x_1^{n_1},x_2^{m_2}x_{12}^{m_{12}}x_1^{m_1})
=\eps_0(\gamma_{\bs\lambda}(x_2^{n_2}x_{12}^{n_{12}}x_1^{n_1})\gamma_{\bs\lambda}(x_2^{m_2}x_{12}^{m_{12}}x_1^{m_1})).
\end{align}

In particular, the non-zero values of $\sigma_{\bs\lambda}$ on the first line are
\begin{center}
\begin{tabular}{|c|c|c|c|c|}
	\hline
	$\sigma$ & $x_2^{N-1}$ & $x_1^{N-1}$ & $x_2x_{12}^{N-1}$ & $x_2^{m-1}x_{12}^{N-m}x_1^{m}, \quad m\in \I_{N-1}$ 
	\\
	\cline{1-5}
	$x_2$ & $\lambda_2$ & $0$ & $0$ & $-(1-q^{N-1})^{N-m}q_{21}^{{(N+m-1)(N-m)}/{2}}\lambda_1\lambda_2$
	\\\cline{1-5}
	$x_1$ & $0$ & $\lambda_1$ & $\lambda_{12}$ & $0$
	\\
	\hline
\end{tabular}
\end{center}\qed
\end{theorem}
A full list of values can be found using \texttt{GAP}, by means of our code 
\href{https://github.com/JoseIgnacio25/Hopf-cocycles-of-Cartan-type-A2/blob/main/generic.g}{generic.g}. 
To do so, choose the order $N=\ord(q)\in\N$, run the script, and for a fixed 6-tuple $(n_2,n_{12},n_1,m_2,m_{12},m_1)=(a,b,c,d,e,f)\in \I_{N-1}^{\circ\,\times6}$  the value \eqref{eqn:cocycle-formula} can be computed with the following command:
\begin{lstlisting}
gap> tigma(a,b,c,d,e,f);
\end{lstlisting}

In this file, we uploaded the computations leading to these values, namely the scalar $\eps_0(\gamma_{\bs\lambda}(x_2^{n_2}x_{12}^{n_{12}}x_1^{n_1})\gamma_{\bs\lambda}(x_2^{m_2}x_{12}^{m_{12}}x_1^{m_1}))$. We obtain it using the definition of $\gamma_{\bs\lambda}$ in Corollary \ref{cor:gamma-formula-QS-no-deformed}, the  formula in Proposition \ref{pro:product}, and the identities in Corollary \ref{cor:eps0}.

		\section{Hopf algebras of Cartan type $A_2$: atypical liftings}\label{sec:N3}
		
		In this section we consider the atypical case, namely when the quantum Serre relations are deformed. In particular, $q\in\G_3'$ and the braiding matrix is $\qb= \begin{psmallmatrix}
			q&q\\q&q
		\end{psmallmatrix}$. Again, we fix a YD-realization $(g_i,\chi_i)_{i\in\I_2}$ of $\qb$ over a semisimple Hopf algebra $H$.
		In the setting \eqref{eqn:coaction-of-E}, we fix $\hat{\Bq}=T(V)$. We compute the Hopf 2-cocycles in \S\ref{sec:hopf3}. We denote by $U_H(\Bq_\qb^\ast)$ the group of convolution units of $\Bq_\qb$ that are $H$-linear.
		
In \S\ref{sec:expo}, we determine which 2-cocycles are pure, that is, not cohomologous to an exponential $e^\eta$ of a Hochschild 2-cocycle $\eta\in \Z^2(\Bq_\bq,\k)^H$. 
%

\subsection{The Hopf cocycles}\label{sec:hopf3}
		
%

We first check that  Hypotheses \textbf{H1} and \textbf{H2} hold.
		\begin{lemma}\label{lem:Hypotheses-deformed-cartan-A2}
		The Hypotheses \textbf{H1} and \textbf{H2} are verified.
		\end{lemma}
		\pf
As in \S\ref{sec:copro}, we begin by computing the coproduct $\Delta\colon T(V)\to T(V)\ot T(V)$, restricted to the elements $x_2^{n_2}x_{12}^{n_{12}}x_1^{n_1}\in T(V)$, $0\leq n_2,n_{12},n_1\leq2$, {\it cf.}~\eqref{eqn:iota}.
		
Now, when $n_{12}<2$, it is straightforward that formula \eqref{eqn:coproduct-preNichols} holds. In turn, if $n_{12}=2$, we get that $\Delta(x_2^{n_2}x_{12}^{2}x_1^{n_1})$ is given by expression \eqref{eqn:coproduct-preNichols} plus an extra term
\[
A\coloneqq \sum
\limits_{\mathclap{\substack{0\leq j\leq n_2\\0\leq m\leq n_1}}}
A(j,m)\,x_2^jx_{112}x_1^m\ot x_2^{n_2-j+1}x_1^{n_1-m},
\]
for $A(j,m)=(1-q^2)q_{12}^{-1}\binom{n_2}{j}_q\binom{n_1}{m}_qq^{n_2-j}q_{21}^{(n_2-j)(m+2)+m}$.
Since Lemma \ref{lem:B-verifies-H1-H2} shows that the (linear) identities \textbf{H1} and \textbf{H2} hold for \eqref{eqn:coproduct-preNichols}. Hence in this context these hypotheses become equivalent, respectively, to 
		\begin{align}\label{eqn:H1-for-N-3}
	((\eps_0\ot\ell)\rho\ot \id)(\pi_{\mE}\tau_2\ot \pi_{\Bq})(A)&=(\pi_{\mE}\tau_2\ot \pi_{\Bq})(A);\\
	\label{eqn:H2-for-N-3}
	[\eps_0\pi_{\mE}\tau_2\ot ((\pi_{\mE}\tau_2\ot \pi_{\Bq})\underline{\Delta}\tau_1)](A)&=0.
\end{align}

Identity \eqref{eqn:H1-for-N-3} holds, as both sides coincide with the expression:
		\begin{align*}
			\lambda_{112}\sum\limits_{\mathclap{\substack{n_2-1\leq j\leq n_2\\0\leq m\leq n_1}}} \ A(j,m)\,y_2^jy_1^m\ot x_2^{n_2-j+1}x_1^{n_1-m}.
		\end{align*}
		Now \eqref{eqn:H2-for-N-3} is immediate from the fact that $(\eps_0\pi_{\mE}\tau_2\ot\id)(A)=A(0,0)\lambda_{112}\,1_{\mE}\ot x_2^{n_2+1}x_1^{n_1}$ and  $((\pi_{\mE}\tau_2\ot \pi_{\Bq})\underline{\Delta})(\tau_2(x_2^{n_2+1}x_1^{n_1}))=0$.		
		Thus the lemma follows.
		\epf

Now we can define a section using \eqref{eqn:general-section}. 

		\begin{corollary}\label{cor:section-QS-deformed}
			The linear map $\gamma_{\bs\lambda}\colon\Bq_\bq\to \mathcal{E}_{\bs\lambda}$ given in the basis $\mathbb{B}_\bq$ for 
			\begin{align}\label{eqn:gamma3}
				\gamma_{\bs\lambda}(x_2^{n_2}x_{12}^{n_{12}}x_1^{n_1})&= y_2^{n_2}y_{12}^{n_{12}}y_1^{n_1}-(\eps_0\ot\ell)(\pi_{\mE}\tau_2\ot \pi_{\Bq})\underline{\Delta}(x_2^{n_2}x_{12}^{n_{12}}x_1^{n_1})
			\end{align}
			is a section in $H\mod$ so that $\eps_0\gamma_{\bs\lambda}=\eps$.
			Explicitly, for $\gamma=\gamma_{\bs\lambda}$:
			\begin{align*}
				&\gamma(x_{12}^2)= y_{12}^2-(q^2-q)\lambda_{112}y_2, \quad 
				\gamma(x_{12}^2x_1)= y_{12}^2y_1+3q^2\lambda_1y_2^2-(q^2-q)\lambda_{112}y_2y_1,\\
				&\gamma(x_{12}x_1^2)= y_{12}y_1^2-(q^2-q)\lambda_1y_2, \quad 
				\gamma(x_2x_{12}^2)= y_2y_{12}^2-(q^2-q)\lambda_{112}y_2^2,\\
								&\gamma(x_2x_{12}x_1^2)= y_2y_{12}y_1^2-(q^2-q)\lambda_1y_2^2,\quad \gamma(x_2x_{12}^2x_1)= y_2y_{12}^2y_1-(q^2-q)\lambda_{112}y_2^2y_1,\\
				&\gamma(x_{12}^2x_1^2)= y_{12}^2y_1^2-3q\lambda_1y_2^2y_1-(q-q^2)\lambda_1y_2y_{12}-(q^2-q)\lambda_{112}y_2y_1^2,\\
				&\gamma(x_2x_{12}^2x_1^2)= y_2y_{12}^2y_1^2-(q-q^2)\lambda_1y_2^2y_{12}-(q^2-q)\lambda_{112}y_2^2y_1^2,\\
				&\gamma(x_2^{n_2}x_{12}^{n_{12}}x_1^{n_1})=y_2^{n_2}y_{12}^{n_{12}}y_1^{n_1}, \text{ in all other cases.}
			\end{align*}
		\end{corollary}
		\pf
		Theorem \ref{thm:H12} defines a (convolution invertible) comodule map $\gamma_{\bs\lambda}:\Bq_\bq\to\mE_{\bs\lambda}$ as in \eqref{eqn:general-section} with $\eps_0\gamma_{\bs\lambda}=\eps_{\Bq_\bq}$. We check on the basis that $\gamma_{\bs\lambda}$ is $H$-linear. 
		\epf
We can apply Theorem \ref{thm:H12} to reach our result, which complements Theorem \ref{thm:cocycle-generic}.
\begin{theorem}\label{thm:cocycle-quantum-serre-deformed}
Let $A$	be a lifting of $\Bq_{\qb}$ over a semisimple Hopf algebra $H$, with braiding matrix $\qb=\begin{psmallmatrix}
	q&q\\q&q
\end{psmallmatrix}$, $q\in\mathbb{G}_3'$.
Then there exists $\bs\lambda=(\lambda_1,\lambda_2,\lambda_{12},\lambda_{112},\lambda_{122})\in\k^5$ such that $A\simeq (\Bq_{\qb}\#H)_{\sigma}$, for $\sigma=\sigma_{\bs\lambda}\#\eps$ and formula \eqref{eqn:cocycle-formula} holds, where $\gamma_{\bs\lambda}$ is as in \eqref{eqn:gamma3}.
	The nonzero values {\it in the first rows} are in the table from Example \ref{exa:intro}
\qed
\end{theorem}
We compute the complete set of values for $\sigma_{\bs\lambda}$ using \texttt{GAP} as in Section \ref{sec:general}. Moreover, for this particular case, we provide a code \href{https://github.com/JoseIgnacio25/Hopf-cocycles-of-Cartan-type-A2/blob/main/atypical.g}{atypical.g}, which computes the {\it orbit} $\alpha\rightharpoonup\sigma_{\bs\lambda}$ of the 2-cocycle evaluated on each pair of basic elements, for any $\alpha\in U_H(\Bq_\bq^\ast)$. See \S\ref{sec:stepii} for further details.
				
		\subsection{The exponential}\label{sec:expo}
	
	In this part, we investigate which cocycles $\sigma\in\Z^2(\Bq_\bq)^H$ can be realized as an exponential $e^\eta$, for some Hochschild 2-cocycle $\eta\in\Z^2(\Bq_\bq,\k)^H$.
		
	Our strategy is as follows.	The input is $\{\sigma(b,b'):b,b'\in\mathbb{B}_\bq\}$, the collection of values  found in Theorem \ref{thm:cocycle-quantum-serre-deformed}; this is represented via a table $T_\sigma$. 
	
	We carry out the following steps:
		\begin{enumerate}[leftmargin=*]
			\item[(i)] In \S\ref{sec:stepi} we compute the group of $H$-linear convolution units $U_H(\Bq_\bq^\ast)$.
			\item[(ii)] In \S\ref{sec:stepii} we calculate the orbit $[\sigma]=\{(\alpha\rhu\sigma)(b,b'):b,b'\in\mathbb{B}_\bq,\alpha\in U\}$. 
			\item[(iii)] In \S\ref{sec:stepiii} we determine the set $\Z^2_0(\Bq_\bq,\k)^H$ of $H$-invariant Hochschild 2-cocycles.
			\item[(iv)] In \S\ref{sec:stepiv} we compute the set  $\{e^{\eta}(b,b'): b,b'\in\mathbb{B}_\bq\}$, for any $\eta\in\Z^2_0(\Bq_\bq,\k)^H$, noting that this does not distinguish coboundaries.
		\end{enumerate}
		
		The output consists of two tables. The first, $T'_\sigma$, obtained in (ii) (see Theorem~\ref{thm:cocycle-quantum-serre-deformed}), encodes all possible Hopf 2-cocycles on $\Bq_\bq$. The second, $T_\eta$, obtained in (iv) (see Lemma~\ref{lem:first_line-e-eta}), represents all possible exponential maps $\Bq_{\bq}\ot\Bq_{\bq} \to \k$ (not only Hopf cocycles).

Comparing $T'_\sigma$ and $T_\eta$ yields the following result, which we prove in \S\ref{sec:proof}.
		
		
		\begin{proposition}\label{pro:Conclusion}
			Let $\sigma_{\bs\lambda}\in Z^2(\Bq_\bq)^H$ be a Hopf cocycle as in Theorem \ref{thm:cocycle-quantum-serre-deformed}. Then
			$\sigma_{\bs\lambda}$ is an exponential Hopf 2-cocycle if and only if either:
			\begin{enumerate}[leftmargin=*]
				\item[(a)] $\lambda_{112}=\lambda_{122}=0$ and at most a single parameter $\lambda_p\in\{\lambda_1,\lambda_2,\lambda_{12}\}$ is non-zero.
				\item[(b)] $\lambda_{112}\lambda_{122}\neq 0$ and $\lambda_1=\tfrac{1}{3}\tfrac{\lambda_{112}^2}{\lambda_{122}}$, $\lambda_2=\tfrac{1}{3}\tfrac{\lambda_{122}^2}{\lambda_{112}}$, $ \lambda_{12}=\tfrac{1}{3}(q^2-q)\lambda_{112}\lambda_{122}$.
			\end{enumerate} 
		\end{proposition}
		
		\begin{remark}
			The mono-parametric 2-cocycles $\sigma_{(0,0,0,1,0)}$ and $\sigma_{(0,0,0,0,1)}$ are pure.
		\end{remark}

%
%
%

		\subsubsection{Step (i)}\label{sec:stepi}
		We compute the $H$-linear morphisms $\alpha:\Bq_\bq\to \k$ with $\alpha(1)=1$. Let  $\{\delta_{x_2^{n_2}x_{12}^{n_{12}}x_1^{n_1}}:0\leq n_2, n_{12}, n_1\leq 2\}$ denote the dual basis of $\mathbb{B}_\bq$. We write
		\begin{align}\label{eqn:alph}
			\alpha=\sum_{\mathclap{0\leq n_2, n_{12}, n_1\leq 2}}\alpha_{(n_2, n_{12}, n_1)}\delta_{x_2^{n_2}x_{12}^{n_{12}}x_1^{n_1}}, \qquad \alpha_{(n_2, n_{12}, n_1)}\coloneqq\alpha(x_2^{n_2}x_{12}^{n_{12}}x_1^{n_1}).
		\end{align}

		\begin{lemma}\label{lem:alpha-definition}
			Let $\alpha\in U_H(\Bq_\bq^\ast)$ with $\alpha(1)=1$. Then
			\begin{align}
			\begin{split}\label{eqn:definition of alpha}
			\alpha=&\eps+\alpha_{(1,1,0)}\delta_{x_2x_{12}}+\alpha_{(2,0,1)}\delta_{x_2^2x_1}+\alpha_{(1,0,2)}\delta_{x_2x_1^2}+\alpha_{(0,1,1)}\delta_{x_{12}x_1}\\	&+\alpha_{(2,1,2)}\delta_{x_2^2x_{12}x_1^2}+\alpha_{(1,2,1)}\delta_{x_2x_{12}^2x_1}+\alpha_{(2,2,0)}\delta_{x_2^2x_{12}^2}+\alpha_{(0,2,2)}\delta_{x_{12}^2x_1^2}.
						\end{split}				\end{align}
The convolution inverse satisfies $\alpha^{-1}(1)=1$, $\alpha^{-1}(b)=-\alpha(b)$,  $\deg(b)=3$, and:
\begin{align*}
	\alpha^{-1}(x_2^2x_{12}x_1^2)&=-\alpha_{(2,1,2)}+\alpha_{(1,0,2)}\alpha_{(1,1,0)} + \alpha_{(2,0,1)}(\alpha_{(0,1,1)}+(q-q^2)\alpha_{(1,0,2)}),\\
	\alpha^{-1}(x_{12}^2x_1^2)&=-\alpha_{(0,2,2)}+(q-q^2)\alpha_{(0,1,1)}\alpha_{(1,0,2)}+\alpha_{(0,1,1)}^2,\\
	\alpha^{-1}(x_2^2x_{12}^2)&=-\alpha_{(2,2,0)}+(q-q^2)\alpha_{(2,0,1)}\alpha_{(1,1,0)}+\alpha_{(1,1,0)}^2,\\
	\alpha^{-1}(x_2x_{12}^2x_1)&=-\alpha_{(1,2,1)}-3\alpha_{(2,0,1)}\alpha_{(1,0,2)}+(q-q^2)\alpha_{(2,0,1)}\alpha_{(0,1,1)}\\
	&\qquad+(q-q^2)\alpha_{(1,0,2)}\alpha_{(1,1,0)}+\alpha_{(1,1,0)}\alpha_{(0,1,1)}.
\end{align*}
		\end{lemma}
		\pf
		Recall that $\alpha$ is $H$-linear if and only if, 	for each $h\in H$,
		\begin{align*}
			\eps(h)\alpha(x_2^{n_2}x_{12}^{n_{12}}x_1^{n_1})
			=\chi_2^{n_2}(h_{(1)})\chi_{2}^{n_{12}}(h_{(2)})\chi_{1}^{n_{12}}(h_{(2)})\chi_1^{n_1}(h_{(3)})\alpha(x_2^{n_2}x_{12}^{n_{12}}x_1^{n_1}),
		\end{align*}
		In particular, if $\alpha(x_2^{n_2}x_{12}^{n_{12}}x_1^{n_1})\neq 0$ and $h\in\{g_1,g_2\}$, then $q^{n_2+2n_{12}+n_1}=1$, that is, $n_2 + 2n_{12} + n_1 \equiv 0 \pmod{3}$. Listing all such tuples with $0\leq n_2, n_{12}, n_1\leq 2$ we obtain the expression in \eqref{eqn:definition of alpha}.
The formulas for $\alpha^{-1}$ follow directly.
		\epf

		\subsubsection{Step (ii)}\label{sec:stepii}
		The calculation of the orbit $[\sigma_{\bs\lambda}]\coloneqq U_H(\Bq_\qb^\ast)\rightharpoonup\sigma_{\bs\lambda}$ is performed using \texttt{GAP}. For any $\alpha$ as in \eqref{eqn:definition of alpha}, we use Theorem \ref{thm:cocycle-quantum-serre-deformed} to compute the values
		\[
\alpha\rightharpoonup\sigma(w,z)=\alpha(w_{(1)})\alpha(z_{(1)})\eps_0(\gamma_{\bs\lambda}(w_{(2)})\gamma_{\bs\lambda}(z_{(2)}))\alpha^{-1}(w_{(3)}z_{(3)})
		\]
		for each $w=x_2^{a}x_{12}^{b}x_1^{c}$, $z=x_2^{d}x_{12}^{e}x_1^{f}\in\BB_\bq$. 
		The full set of values of the orbit is obtained in \href{https://github.com/JoseIgnacio25/Hopf-cocycles-of-Cartan-type-A2/blob/main/atypical.g}{atypical.g} via:
\begin{lstlisting}
gap> osigma(a,b,c,d,e,f);
\end{lstlisting}

In particular, we have the following modification of the table in Theorem \ref{thm:cocycle-quantum-serre-deformed}:
		\begin{lemma}\label{lem:orbit}
The nonzero values in the first rows of $[\sigma_{\bs\lambda}]$ are given by:
	\begin{center}
		{\begin{tabular}{|c|c|c|c|c|c|c|c|c|}
				\hline
				$[\sigma_{\bs\lambda}]$  & $x_2^2$ & $x_2x_1$ & $x_{12}$ & $x_1^2$ & $x_2^2x_{12}x_1$ & $x_2x_{12}^2$ & $x_2x_{12}x_1^2$& $x_{12}^2x_1$  \\
				\hline
				$x_2$  & $\ast$ & $T_{1,2}$ & $T_{1,3}$ & $T_{1,4}$ & $T_{1,5}$ & $\ast+T_{1,6}$ & $\ast+T_{1,7}$ & $ \ast+T_{1,8}$ \\
				\hline
				$x_1$  & $\ast+T_{2,1}$ & $ T_{2,2}$ & $\ast+T_{2,3}$ & $\ast$ & $T_{2,5}$ & $\ast+T_{2,6}$ & $\ast+T_{2,7}$ & $\ast+T_{2,8}$\\
				\hline
		\end{tabular}}
	\end{center}
Here $\ast$ stands for value of $\sigma_{\bs\lambda}$ in Theorem \ref{thm:cocycle-quantum-serre-deformed} and the scalars $T_{i,j}\in\k$ are:
\begin{align*}
	&T_{1,2}=-\alpha_{(2,0,1)},\qquad T_{1,3}=-\alpha_{(1,1,0)},\qquad\qquad\quad\quad\ \, T_{2,1}=-q^2\alpha_{(2,0,1)}+\alpha_{(1,1,0)},\\
	& T_{1,4}=-\alpha_{(1,0,2)},\qquad T_{2,2}=-q\alpha_{(1,0,2)}-\alpha_{(0,1,1)},\qquad T_{2,3}=-q^2\alpha_{(0,1,1)},\\
	&T_{1,5}=(q-q^2)(\lambda_2\alpha_{(1,0,2)}+\alpha_{(2,0,1)}\alpha_{(1,1,0)})+(q^2-1)\alpha_{(2,0,1)}^2,\\
	&T_{1,6}=-3\lambda_2\alpha_{(1,0,2)}+(q-q^2)(\lambda_2\alpha_{(0,1,1)}+\alpha_{(2,0,1)}\alpha_{(1,1,0)})-q\alpha_{(1,1,0)}^2-\alpha_{(2,2,0)},\\
	&T_{1,7}=-\alpha_{(2,0,1)}(3q^2\alpha_{(1,0,2)}+q\alpha_{(0,1,1)})-q\alpha_{(1,0,2)}\alpha_{(1,1,0)}-\alpha_{(2,1,2)},\\
	&T_{1,8}=\alpha_{(1,0,2)}(-3\alpha_{(2,0,1)}+(q-q^2)\alpha_{(1,1,0)})-q^2\alpha_{(1,1,0)}\alpha_{(0,1,1)}-\alpha_{(1,2,1)},\\
	&T_{2,5}=(q^2-q)(\lambda_{112}\alpha_{(2,0,1)}-\lambda_{122}\alpha_{(1,0,2)})+\alpha_{(2,0,1)}((1-q)\alpha_{(1,0,2)}+2q^2\alpha_{(0,1,1)})\\
	&\qquad\qquad-\alpha_{(1,1,0)}(\alpha_{(1,0,2)}-q\alpha_{(0,1,1)})-q\alpha_{(2,1,2)}+\alpha_{(1,2,1)},\\
	&T_{2,6}=(q-q^2)(\lambda_{112}\alpha_{(1,1,0)}+\lambda_{122}\alpha_{(0,1,1)})-\alpha_{(1,0,2)}(3\lambda_{122}+(1-q^2)\alpha_{(1,1,0)})\\
	&\qquad\qquad+2q\alpha_{(1,1,0)}\alpha_{(0,1,1)}-q^2\alpha_{(1,2,1)},\\
	&T_{2,7}=(q^2-q)(\lambda_1\alpha_{(2,0,1)}-\lambda_{112}\alpha_{(1,0,2)})-\alpha_{(1,0,2)}((1-q)\alpha_{(1,0,2)}+3q^2\alpha_{(0,1,1)})\\
	&\qquad\qquad-q\alpha_{(0,1,1)}^2-\alpha_{(0,2,2)},\\\
	&T_{2,8}=3\lambda_1\alpha_{(2,0,1)}+(q^2-q)\lambda_1\alpha_{(1,1,0)}+(q^2-q)\lambda_{112}\alpha_{(0,1,1)}-q\alpha_{(0,1,1)}^2-q\alpha_{(0,2,2)}.
\end{align*}
\qed
\end{lemma}
\begin{remark}\label{rem:notation-alphas}
To interpret or verify the values of $[\sigma_{\boldsymbol{\lambda}}]$—in particular, those displayed in the table of Lemma~\ref{lem:orbit}— using the script \href{https://github.com/JoseIgnacio25/Hopf-cocycles-of-Cartan-type-A2/blob/main/atypical.g}{atypical.g}, we clarify that the following notation was used in the code for the variables: 
\begin{align*}
l_i&=\lambda_i, i\in \I_2; &l_{12}&=\lambda_{12}; &l_{112}&=\lambda_{112}; &l_{122}&=\lambda_{122};\\
\alpha_{(2,0,1)}&=a_{221}; &\alpha_{(1,0,2)}&=a_{211};&\alpha_{(0,1,1)}&=a_{121};&\alpha_{(1,1,0)}&=a_{212};\\
\alpha_{(2,1,2)}&=a_{221211}; &\alpha_{(0,2,2)}&=a_{121211};&\alpha_{(1,2,1)}&=a_{212121}; &\alpha_{(2,2,0)}&=a_{221212}.
\end{align*}
Recall that the scalars $\alpha_{(n_2, n_{12}, n_1)}$ were introduced in \eqref{eqn:alph}.
\end{remark}

\subsubsection{Step (iii)}\label{sec:stepiii}
We describe $\Z^2_0(\Bq_\bq,\k)^H$ in Corollary \ref{cor:z2}. 
We begin by computing the subset of coboundaries in $\Z^2_0(\Bq_\bq,\k)^H$, namely, $\B_0^2(\Bq_\bq,\k)^H$. For $b\in\BB_\bq$, define the $\k$-linear map 
\[
\beta_{b}:\Bq_\bq\ot \Bq_\bq\to\k, \qquad \beta_{b}(r,s)=-\delta_{b}(rs), \ r,s\in \Bq_\bq^+
\]
and $\beta_{b}(r,1)=\beta_b(1,s)=0$, for any $r,s\in \Bq_\bq$. 
\begin{lemma}\label{lem:coboundaries}
The set 
$			\{\beta_{x_2x_{12}},\beta_{x_2^2x_1},\beta_{x_2x_1^2},\beta_{x_{12}x_1},\beta_{x_2^2x_{12}x_1},\beta_{x_2x_{12}^2x_1^2},\beta_{x_2^2x_{12}^2},\beta_{x_{12}^2x_1^2}\}
$
is a basis for $\B_0^2(\Bq_\bq,\k)^H$.
\end{lemma} 
\pf
By definition, each $\beta\in \B^2_0(\Bq_\bq,\k)^H\subset\Z_0^2(\Bq_\bq,\k)^H$ corresponds to a $H$-linear map $f=f_\beta:\Bq_\bq\to\k$ such that $f(1)=0$, so $\beta(r,s)=-f(rs)$ for $r,s\in\Bq_\bq^+$. 	In particular $\beta_{b}\in\B_0^2(\Bq_\bq,\k)$ for each $b\in\BB_\bq$. The result follows from Lemma \ref{lem:alpha-definition}.
\epf

Next we compute the cocycles. Following  \cite[\S 6.1]{MW}, we consider the Hochschild 2-cocycles $\eta_1,\eta_2,\eta_{12}$ associated with the ideal $\lg x_1^3,x_2^3,x_{12}^3\rg$ of relations in $\tilde{\Bq}$, {\it cf.}~\S\ref{sec:liftA2}: 
			\begin{align*}
		\eta_1(r, s)=\begin{cases}
			1, \text{ if } r=x_{1}^{n_1}, s=x_{1}^{3-{n_1}},\\
			0, \text{ otherwise;}
		\end{cases} 			\eta_2(r, s)=\begin{cases}
			1,\text{ if }r=x_{2}^{n_2}, s=x_{2}^{3-n_2},\\
			0,\text{ otherwise;}
		\end{cases}\\
		\eta_{12}(r, s)=\begin{cases}
			1, &\text{if } r=x_{12}^{n_{12}}, s=x_{12}^{3-{n_{12}}},\\
			({n_1})!_qq_{12}^{\frac{{n_1}^2-{n_1}}{2}}, & \text{if } r=x_{12}^{n_{12}}x_1^{n_1}, s=x_2^{n_1}x_{12}^{{m_{12}}}, {n_{12}}+{m_{12}}+{n_1}=3,\\
			0,& \text{otherwise.}
		\end{cases}
	\end{align*}
A straightforward computation shows that $\eta_i\in\Z^2(\Bq_\bq,\k)^H$ if and only if $\chi_i^N=\eps$, $i\in\I_2$, while 
$\eta_{12}\in\Z^2(\Bq_\bq,\k)^H$ if and only if $(\chi_1\chi_2)^N=\eps$; see  \cite[Theorem 6.1.3 and Theorem 6.1.4]{MW}.

We now introduce two additional invariant, non-trivial 2-cocycles, associated to the quantum Serre relations.
\begin{definition}
	Let $I_{112}, I_{122}\subset \BB_\bq\times \BB_\bq$ be the subsets
$I_{112}=\{(x_1^2,x_2),(x_1,x_{12})\}$ and $I_{122}=\{(x_1,x_2^2),(x_{12},x_2)\}$.	
	We let  $\eta_{112},\eta_{122}:\Bq_\bq\ot \Bq_\bq\to \k$ be the linear maps determined by, for each $(b,b')\in\BB_\bq\times \BB_\bq$:
	\begin{align*}
		\eta_{112}(b,b')&=\begin{cases}
			1\quad \text{if $(b,b')\in I_{112}$},\\
			0 \quad\text{elsewhere};
		\end{cases} &&
		\eta_{122}(b,b')=\begin{cases}
			1\quad \text{if $(b,b')\in I_{122}$},\\
			0 \quad \text{elsewhere}.
		\end{cases}
	\end{align*}
\end{definition}
\begin{lemma}\label{lem:eta112-eta122}
The maps $\eta_{112}$ and $\eta_{122}$ are Hochschild 2-cocycles for $\Bq_\bq$. Moreover, 
\begin{itemize}
\item $\eta_{112},\eta_{122}\notin \B^2(\Bq_{\bq},\k)$.
\item $\eta_{112}\in\Z^2(\Bq_\bq,\k)^H$ iff $\chi_1^2\chi_2=\eps$ and $\eta_{122}\in\Z^2(\Bq_\bq,\k)^H$ iff $\chi_1\chi_2^2=\eps$.
\end{itemize}
\end{lemma}	
\pf
We treat $\eta = \eta_{112}$; the proof for $\eta_{122}$ is analogous.
Since $\eta$ is concentrated in degree 3, namely, $\eta_{|(\Bq_\bq\ot\Bq_\bq)_k}=0$ if $k\neq 3$, then we only need to show that $\eta(ab,c)=\eta(a,bc)$ for $a,b,c\in\Bq_\bq^+$ such that $\deg(abc)=3$, that is when $a,b,c\in V$. The only case where $\eta(ab,c)$ and $\eta(a,bc)$ are not zero is $a=b=x_1$ and $c=x_2$, moreover it is clear that $\eta(x_1^2,x_2)=\eta(x_1,x_1x_2)$. This shows $\eta\in\Z^2(\Bq_\bq,\k)$.

Next we show that $\eta_{112}, \eta_{122}\notin\B^2(\Bq_{\bq},\k)$. Assume there exists a linear map $f:\Bq_\bq\to\k$ such that $\eta_{122}(a,b)=f(ab)$, for any $a,b\in\Bq_\bq^+$. This gives $\eta_{122}(x_1,x_2^2)=0$, which is a contradiction. Indeed:
\begin{align*}
	\eta_{122}(x_1,x_2^2)&=f(x_1x_2^2)=f(x_{12}x_2+q_{12}x_2x_1x_2)\\
				&\stackrel{(\ast)}{=}q_{12}(1+q_{22})f(x_2x_{12})+q_{12}^2f(x_2^2x_1)\\
	&=q_{12}(1+q_{22})\eta_{122}(x_2,x_{12})+q_{12}^2\eta_{122}(x_2,x_2x_1)=0.
\end{align*}
Equality $(\ast)$ follows from the quantum Serre relation $x_{122}=0$ and the definition of $x_{12}$.
The proof for $\eta_{112}$ is analogous.

The last statement is clear since $h\cdot x_i=\chi_i(h)x_i$, for each $i\in\I_2$, $h\in H$.
\epf

		\begin{proposition}
The set of cohomology classes $\{[\eta_1],[\eta_2], [\eta_{12}], [\eta_{112}], [\eta_{122}]\}$ is a $\k$-linear basis of $\H^2(\Bq_\bq,\k)$.
		\end{proposition}
		\pf
By \cite[Lemma 5.4]{VK}, we have that $\dim \H^2(\Bq_\bq,\k)=5$.
Suppose there exists $\beta\in\B^2(\Bq_{\bq},\k)$ and $c_{1},c_{2},c_{12}, c_{112}, c_{122}\in\k$ such that
\begin{align}\label{eqn:beta-etas}
\beta=c_1\eta_1+c_2\eta_2+c_{12}\eta_{12}+c_{112}\eta_{112}+c_{122}\eta_{122}.
\end{align}
Since $\beta(r,s)=f(rs)$ for some linear map $f\colon\Bq_\qb\to \k$ and all $r,s\in \Bq_\qb^+$, we can conveniently evaluate \eqref{eqn:beta-etas} to deduce that $c_{1}=c_{2}=c_{12}=c_{112}=c_{122}=0$. Indeed, we first observe that $0=f(x_i^3)=\beta(x_i,x_i^2)=c_i$ for each $i\in \I_2$, and similarly $0=f(x_{12}^3)=\beta(x_{12},x_{12}^2)=c_{12}$. 
Next, evaluating \eqref{eqn:beta-etas} at $(x_2,x_{12})$, we obtain $0=\beta(x_2,x_{12})$, while $\beta(x_{12},x_2)=q_{22}q_{12}\beta(x_2,x_{12})$; hence $c_{122}=\beta(x_{12},x_2)=0$. By a similar argument, we deduce that $c_{112} = 0$.
		\epf
		
		\begin{corollary}\label{cor:z2}
			Any $\eta\in\Z_0^2(\Bq_\bq,\k)^H$ can be written as
			\begin{align}\label{eqn:general-eta}
				\eta=e_{1}\eta_1+e_{2}\eta_2+e_{12}\eta_{12}+e_{112}\eta_{112}+e_{122}\eta_{122}+\beta,
			\end{align}
			for some $e_1,e_2,e_{12},e_{112},e_{122}\in\k$ and $\beta\in\B_0^2(\Bq_\bq,\k)^H$.
		\end{corollary}
\begin{remark}
Observe that $\Z^2(\Bq_\bq,\k)^H=\k\eps+\Z^2_0(\Bq_\bq,\k)^H$.
\end{remark}

\subsubsection{Step (iv)}\label{sec:stepiv}
We compute the values of $e^\eta$ via \texttt{GAP}, via the file \href{https://github.com/JoseIgnacio25/Hopf-cocycles-of-Cartan-type-A2/blob/main/exponential.g}{exponential.g}. Specifically, if  $w=x_2^{a}x_{12}^{b}x_1^{c}$ and $z=x_2^{d}x_{12}^{e}x_1^{f}$, then
\begin{lstlisting}
gap> exponential(a,b,c,d,e,f);
\end{lstlisting}
returns the value $e^\eta(w,z)$, for any $a,b,c,d,e,f\in \I_2^{\circ\times 6}$ and $\eta$ as in \eqref{eqn:general-eta}.

\begin{lemma}\label{lem:first_line-e-eta}
	The nonzero values in the first rows of $e^\eta$ are as follows:

\noindent	
	\resizebox{12.6cm}{!}{
	\begin{tabular}{|c|c|c|c|c|c|c|c|c|}
		\hline
		$e^\eta$ & $x_2^2$ & $x_2x_1$ & $x_{12}$ & $x_1^2$ & $x_2^2x_{12}x_1$ & $x_2x_{12}^2$ & $x_2x_{12}x_1^2$ & $x_{12}^2x_1$ \\
		\hline
		$x_2$ &
		$e_2$ &
		$-b_{221}$ &
		$-b_{212}$ &
		$-b_{211}$ &
		$0$ &
		$-b_{221212}$ &
		$-b_{221211}$ &
		$-b_{212121}$ \\
		\hline
		$x_1$ &
		\makecell{ $e_{122}$ \\ $-q^2b_{221}$ \\ $+b_{212}$ } &
		\makecell{ $-qb_{211}$ \\ $-b_{121}$ } &
		\makecell{ $e_{112}$ \\ $-q^2b_{121}$ } &
		$e_1$ &
		\makecell{ $b_{212121}$ \\ $-qb_{221211}$ } &
		\makecell{ $e_{12}$ \\ $-q^2b_{212121}$ } &
		$-b_{121211}$ &
		$-q b_{121211}$ \\
		\hline
\end{tabular}}

\noindent
Here, the scalars $b_{\ast} \in \k$ are defined as follows:
\begin{align}
\notag	&-b_{212}=\beta(x_2,x_{12}), \quad-b_{221}=\beta{(x_2,x_2x_1)},\quad-b_{221211}=\beta{(x_2,x_2x_{12}x_1^2)},\\
\label{eqn:betas}	&-b_{212121}=\beta{(x_2,x_{12}^2x_1)},\quad-b_{211}=\beta{(x_2,x_1^2)},\quad-b_{121}=\beta{(x_{12},x_1)},\\
\notag	&-b_{221212}=\beta{(x_2,x_2x_{12}^2)},\quad -b_{121211}=\beta{(x_{12},x_{12}x_1^2)}.
\end{align}
\end{lemma}

\begin{remark}\label{rem:truncated-exponential}
We actually define functions \texttt{exponential2},..., \texttt{exponential5} that compute the value of the {\it truncated} exponential $e_n^\eta\coloneqq\sum_{k=0}^n \frac{\eta^{\ast k}}{k!}$ up to degree $n$. By degree considerations, one has $e^\eta=e_5^\eta$. However, depending on the degrees of the basis elements $w,z\in\BB$, it is often sufficient to use a truncated version to reduce computational cost. For instance, if $\deg(w)+\deg(z)\leq 9$, then $e^\eta(w,z)=e^\eta_3(w,z)$.
\end{remark}

%

\subsection{Proof of Proposition \ref{pro:Conclusion}}\label{sec:proof} 
Suppose that there exist $\alpha\in U_H(\Bq_\qb^\ast)$, $\eta\in\Z_0^2(\Bq_\bq,\k)^H$, as in \eqref{eqn:general-eta}, such that $[\sigma_{\bs\lambda}]\coloneqq´\alpha\rightharpoonup\sigma_{\bs\lambda}=e^\eta$. We fix $\beta\in\B_0^2(\Bq_\bq,\k)^H$ and let $e_{j}\in\k$ be as in loc.cit.
We follow the notation for $[\sigma_{\bs\lambda}]$ introduced in Remark \ref{rem:notation-alphas} and that for $\beta$ introduced in \eqref{eqn:betas}.

We proceed in several steps.
		
		\noindent\textit{Step 1}: From $([\sigma_{\bs\lambda}]-e^\eta)(x_i,b)$, $i\in \I_2$, $b\in \BB_\bq$, we derive the following constraints:
		\begin{align*}
			&l_1=e_1;\qquad l_{2}=e_{2};\qquad l_{112}=e_{112};\qquad l_{122}=e_{122};\\
			&b_{221}=a_{221};\qquad b_{211}=a_{211};\qquad b_{121}=a_{121},\qquad b_{212}=a_{212};\\
			&b_{212121}=a_{212121}-3q^2l_{2}l_{1}+3a_{221}a_{211}+(q^2-q)a_{211}a_{212}+q^2a_{212}a_{121};\\
			&b_{221211}=a_{221211}+(q^2-q)l_2l_1+3q^2a_{221}a_{211}+qa_{221}a_{121}+qa_{211}a_{212};\\
			&b_{121211}=a_{121211}+(q^2-q)(l_{1}l_{122}-l_{1}a_{221}+l_{112}a_{211})+3q^2a_{211}a_{121}+qa_{121}^2\\
			&\qquad\qquad +(1-q)a_{211}^2;\\
			&b_{221212}=a_{221212}+3l_{2}a_{211}+(q^2-q)(l_{2}l_{112}+l_{2}a_{121}+a_{221}a_{212})+qa_{212}^2.
		\end{align*}

			\noindent\textit{Step 2}: 
		We update the previous configurations and continue comparing $\alpha\rightharpoonup\sigma$ and $e^\eta$ for $(a,b)\in\BB_\bq\ot\BB_\bq$ such that $\deg(a)+\deg(b)=3$. Below we list a number of equations along with their corresponding $(a,b)$.
		\begin{align}\label{E1}\tag{E1}
			& (x_2^2, x_2x_{12}x_1): l_{2}a_{211}+a_{221}^2=0.\\\label{E2}\tag{E2}
			&(x_2x_{12}x_1, x_1^2):  l_{1}a_{221}+a_{211}^2=0.\\\label{E3}\tag{E3}
			&(x_2^2, x_{12}^2): l_{2}a_{211}-\tfrac{1}{3}q^2a_{212}^2=0.\\\label{E4}\tag{E4}
			&(x_{12}x_1^2, x_1^2):  (q-1)l_{1}a_{211}+l_{1}a_{121}=0.\\\label{E5}\tag{E5}
			&(x_1^2, x_2x_{12}x_1):  l_{1}l_{122}-q^2l_{1}a_{221}+l_{1}a_{212}-q^2a_{211}^2-2qa_{211}a_{121}-a_{121}^2=0.\\\label{E6}\tag{E6}
			&(x_1^2, x_{12}^2): l_{1}l_{122}-q^2l_{1}a_{221}+l_{1}a_{212}+q^2l_{112}^2-2ql_{112}a_{211}+2q^2l_{112}a_{121}\\\notag
			&\qquad\qquad+2a_{211}^2+2(q^2-q)a_{211}a_{121}-a_{121}^2=0.\\\label{E7}\tag{E7}
			&(x_{12}^2, x_2^2): 3l_{2}l_{112}+3(q-1)l_{2}a_{211}+3l_{2}a_{121}+(q-1)l_{122}^2+2(q-1)l_{122}a_{212}\\\notag
			&\qquad\qquad-3a_{212}^2=0.\\\label{E8}\tag{E8}
			&(x_{12}^2x_1, x_2): 3(q^2-q)l_{2}l_{1}+3l_{112}a_{221}+(q^2-q)l_{112}a_{212}+(q^2-q)l_{122}a_{121}\\\notag
			&\qquad\qquad+3a_{221}a_{121}+(q^2-q)a_{212}a_{121}+l_{12}-e_{12}=0.\\\label{E9}\tag{E9}
			&(x_2^2, x_2^2x_{12}): (q-1)l_{2}a_{221}+l_{2}a_{212}=0.\\\label{E10}\tag{E10}
			&(x_1^2, x_{12}x_1^2): l_{1}l_{112}+(q-q^2)l_{1}a_{211}+2l_{1}a_{121}=0.\\\label{E11}\tag{E11}
			&(x_2^2x_{12}, x_2^2): l_{2}l_{122}+(q-q^2)l_{2}a_{221}+2l_{2}a_{212}=0.\\\label{E12}\tag{E12}
			&(x_2x_{12}^2, x_1):  3l_{2}l_{1}-3qa_{221}a_{211}+(q^2-1)a_{212}a_{121}=0.\\\label{E13}\tag{E13}
			&(x_2^2x_{12}x_1, x_2):  l_{2}l_{112}+2ql_{2}a_{211}+l_{2}a_{121}+l_{122}a_{221}-q^2a_{221}^2-a_{212}^2\\\notag
			&\qquad\qquad+(1-q)a_{221}a_{212}=0.\\\label{E14}\tag{E14}
			&(x_1^2, x_2x_{12}x_1):  l_{1}l_{122}-q^2l_{1}a_{221}+l_{1}a_{212}-q^2a_{211}^2-2qa_{211}a_{121}-a_{121}^2=0.
		\end{align}

		\begin{claim}\label{Claim:in-degree-3}
			The following holds:
			\begin{enumerate}[leftmargin=*]
				\item If $l_1l_2=0$, then $a_{221}=a_{211}=a_{121}=a_{212}=0$, $l_{112}=l_{122}=0$ and $l_{12}=e_{12}$.
				\item If $l_1l_2\neq 0$, then 
				\begin{align}\label{eqn:a221-a211}
					&a_{221}=-\tfrac{l_{122}}{3};&& a_{211}=-\tfrac{l_{112}}{3};\\\label{eqn:a212-a121}
					& a_{212}=\tfrac{(q-1)}{3}l_{122};&& a_{121}=\tfrac{(q-1)}{3}l_{112};\\\label{eqn:l2-l1}
					&l_2=\tfrac{1}{3}\tfrac{l_{122}^2}{l_{112}};&&l_1=\tfrac{1}{3}\tfrac{l_{112}^2}{l_{122}};
					\\\label{eqn:e12-l12}& e_{12}=l_{12}+\tfrac{(q-q^2)}{3}l_{112}l_{122}.&&
				\end{align}
			\end{enumerate}
		\end{claim}
		
		We start by assuming $l_1l_2=0$. On the one hand, if $l_1=0$, then equations \ref{E1}-\ref{E3} and \ref{E5}-\ref{E7} yield  $a_{211}=a_{221}=a_{212}=a_{121}=0$ and $l_{112}=l_{122}=0$. In turn, if $l_1\neq 0$, then \ref{E1}-\ref{E4} and \ref{E6}-\ref{E7} imply $a_{211}=a_{221}=a_{212}=a_{121}=0$ and $l_{112}=l_{122}=0$. Equation \ref{E8} establishes $l_{12}=e_{12}$ in each case. This proves {\it (1)}.
		
		To prove  {\it (2)}, assume $l_1l_2\neq0$. Then, the identities in \eqref{eqn:a221-a211} follow by substituting \ref{E9} into \ref{E11}, and \ref{E4} into \ref{E10}, respectively. The formulas in \eqref{eqn:a212-a121} are obtained by plugging \eqref{eqn:a221-a211} into \ref{E9} and \ref{E4}, respectively.
		Next, we use \eqref{eqn:a221-a211} and \eqref{eqn:a212-a121}
		in \ref{E12} to get $l_2l_1=\tfrac{1}{9}l_{112}l_{122}$.
		Replacing all $a_{221}$, $a_{212}$, $a_{112}$, and $a_{121}$ in \ref{E13} and \ref{E14}, we obtain, respectively, the equalities in \eqref{eqn:l2-l1}.
		Finally, \eqref{eqn:e12-l12} is obtained analogously by replacing $l_1$ and $l_2$ in \ref{E8}. This completes the proof of {\it (2)}.

		If $l_2l_1=\tfrac{1}{9}l_{112}l_{122}$, then $l_{112}l_{122}=0$ implies $l_{1}l_{2}=0$. Hence we land in case {\it (1)}, and thus $l_{112}=l_{122}=0$. Therefore,  $l_{112}l_{122}=0$ if and only if $l_{112}=l_{122}=0$, and $l_{112}l_{122}\neq0$ if and only if $l_1l_2\neq 0$.

\noindent{\it Step 3}: We consider the conditions in each case given by Claim \ref{Claim:in-degree-3} and establish new equations in higher degrees. As above, when $l_1l_2=0$, we analyze separately the cases $(l_1\neq 0, l_2= 0)$, $(l_1= 0,l_2\neq 0)$ and $(l_1= 0,l_2=0)$. 
		\begin{claim}\label{Claim:in-degree-9}
			The following holds:
			\begin{enumerate}[leftmargin=*]
				\item If $l_2\neq 0$ and $l_1= 0$, then $a_{121211}=a_{221211}=a_{212121}=l_{12}=0$.
				\item If $l_1\neq 0$ and $l_2= 0$, then $a_{221212}=a_{221211}=a_{212121}=l_{12}=0$.
				\item If $l_1= 0$ and $l_2=0$, then $a_{212121}=qa_{221211}$.
				\item If $l_1l_2\neq0$ then
				\begin{align}\label{eqn:a212121}
					a_{212121}&=\tfrac{(1-q)}{9}l_{112}l_{122}+qa_{221211}, 
					& a_{221212}&=\tfrac{(q-q^2)}{9}l_{122}^2+\tfrac{l_{122}}{l_{112}}a_{221211},\\\label{eqn:a121211}
					a_{121211}&=\tfrac{(q-q^2)}{9}l_{112}^2+\tfrac{l_{112}}{l_{122}}a_{221211},
					&l_{12}&=\tfrac{(q^2-q)}{3}l_{112}l_{122}.
				\end{align}
			\end{enumerate}
		\end{claim}
	Observe that the last equation implies $e_{12}=0$.
		Assume that $l_1l_2=0$ and recall all the constraints from  Claim \ref{Claim:in-degree-3}{\it (1)}. We write here the most relevant equations:
		\begin{align}\label{E15}\tag{E15}
			&(x_{12},x_2^2x_{12}^2x_1): (q^2-q)l_2a_{121211}=0.\\\label{E16}\tag{E16}
			&(x_2x_{12},x_2^2x_{12}^2): -3l_2a_{221211}+(q^2-q)l_2a_{212121}=0.\\\label{E17}\tag{E17}
			&(x_2x_{12}^2,x_2^2x_{12}): 3q^2l_2l_{12}+3ql_2a_{221211}+(q^2-1)l_2a_{212121}=0.\\\label{E18}\tag{E18}
			&(x_2^2x_1,x_2^2x_{12}^2):  (q^2-q)l_2a_{221211}=0.\\
			\label{E19}\tag{E19}&(x_{12},x_2x_{12}^2x_1^2): (q^2-1)l_1a_{221212}=0.\\\label{E20}\tag{E20}
			&(x_{12}^2x_1^2,x_2x_1^2): (q-q^2)l_1a_{221211}=0.\\\label{E21}\tag{E21}
			&(x_{12}^2x_1^2,x_{12}x_1): 3l_1a_{221211}+(q-q^2)l_1a_{212121}=0.\\\label{E22}\tag{E22}
			&(x_1^2,x_2x_{12}^2x_1^2): (q-q^2)l_1l_{12}+3l_1a_{221211}+2(q-q^2)l_1a_{212121}=0.\\\label{E23}\tag{E23}
			&(x_{12}^2x_1,x_2x_{12}^2x_1^2): ((q^2-1)a_{221211} + (q^2-q)a_{212121})a_{121211}=0.\\\label{E24}\tag{E24}
			&(x_2x_{12}^2,x_2^2x_{12}^2x_1):((q-1)a_{221211}+(q^2-1)a_{212121})a_{221212}=0.\\\label{E25}\tag{E25}
			&(x_2^2x_{12}x_1,x_2x_{12}^2x_1^2): ((q^2-1)a_{221211}+(q^2-q)a_{212121})a_{221211}=0.\\\label{E26}\tag{E26}
			&(x_2x_{12}^2,x_2x_{12}^2x_1^2): ((q^2-1)a_{221211}+(q^2-q)a_{212121})a_{212121}=0.
		\end{align}
		Then {\it (1)}, {\it (2)} and {\it (3)} follow by \ref{E15}-\ref{E18}, \ref{E19}-\ref{E22} and \ref{E23}-\ref{E26}, respectively.

		On the other hand, suppose that $l_1l_2\neq0$. Then, by Claim \ref{Claim:in-degree-3}{\it (2)}, we get:
		\begin{align}\label{E27}\tag{E27}
			&(x_2x_1^2, x_2x_{12}^2x_1): \tfrac{(q-1)}{9}l_{112}l_{122}-qa_{221211}+a_{212121}=0.\\\label{E28}\tag{E28}
			&(x_2^2x_{12}^2, x_{12}x_1):  a_{221211}+\tfrac{(q-q^2)}{3}a_{212121}+\tfrac{(q-1)}{3}\tfrac{l_{122}}{l_{112}}a_{121211}=0.\\\label{E29}\tag{E29}
			&(x_2x_{12}, x_2^2x_{12}x_1^2): \tfrac{l_{112}}{l_{122}}a_{221212}+(q^2-1)a_{221211}-qa_{212121}=0.\\\label{E30}\tag{E30}
			&(x_{12}x_1^2, x_{12}^2x_1):  l_{112}l_{122}+(q^2-q)l_{12}+(q-1)a_{221211}\\			\notag &\qquad\qquad\qquad
			\qquad\qquad\qquad+q^2a_{212121}		-\tfrac{l_{122}}{l_{112}}qa_{121211}=0.
		\end{align}
		We substitute  $a_{212121}$ in equations \ref{E28} and \ref{E29} with the expression found in \ref{E27}. This gives \eqref{eqn:a212121} and the left identity in \eqref{eqn:a121211}. For the remaining one, we use this new information and replace it in \ref{E30}. Hence, the claim follows. Therefore, we have proved the necessary condition part in Proposition \ref{pro:Conclusion}. 
		
		To prove the sufficient condition part, we use the constraints derived in the {\it First}, {\it Second} and {\it Third} steps to define $\alpha$ and $\eta$ in each case. We then introduce these definitions into the codes \href{https://github.com/JoseIgnacio25/Hopf-cocycles-of-Cartan-type-A2/blob/main/atypical.g}{atypical.g} and \href{https://github.com/JoseIgnacio25/Hopf-cocycles-of-Cartan-type-A2/blob/main/exponential.g}{exponential.g}. Finally, we verify that the tables of $\alpha\rightharpoonup\sigma_{\bs\lambda}$ and $e^\eta$ coincide. \qed
		
\begin{remark}

The proof of Proposition \ref{pro:Conclusion} shows that coboundaries are essential to account for all exponential Hopf 2-cocycles, even up to equivalence. For instance, let $\eta\in\Z_0^2(\Bq_\bq,\k)^H$ be as in \eqref{eqn:general-eta} with coefficients $e_{112}=e_{122}=1$, $e_1=e_2=\frac{1}{3}$ and $e_{12}=0$. Then $e^{\eta}$ defines a Hopf 2-cocycle only if the following identity holds:
\begin{align*}
	\beta=\tfrac{(q-1)}{3}(\beta_{x_2x_{12}}+\beta_{x_{12}x_1})-\tfrac{1}{3}(\beta_{x_2^2x_1}+\beta_{x_2x_1^2})+\kappa(\beta_{x_2^2x_{12}^2}+\beta_{x_{12}^2x_1^2}+q\beta_{x_2x_{12}^2x_1}),
\end{align*}
where $\kappa\coloneqq\beta(x_2,x_2x_{12}x_1^2)\in \k$.
\end{remark}


	\end{document}